\documentclass{article}
\usepackage{bm}
\usepackage{tikz,mathpazo}
\usepackage{pgf}
\usepackage[top=2.5cm, bottom=2.5cm, left=3cm, right=3cm]{geometry}   
\usepackage[normalem]{ulem}
\usepackage{enumitem}
\usepackage{indentfirst}
\usepackage{graphicx}
\usepackage{graphics}
\usepackage[toc,page,title,titletoc,header]{appendix}
\usepackage{bm}
\usepackage{bbm}
\usepackage{listings}  
\usepackage{amsmath}
\usepackage{setspace} 
\usepackage{indentfirst}
\usepackage{caption}
\usepackage{multirow} 
\usepackage{lipsum,multicol}
\usepackage{pdfpages}
\usepackage{float}
\usepackage{amsthm}
\usepackage{pdfpages}
\usepackage{url}
\usepackage{colortbl}
\usepackage{subfigure}
\usepackage{epsfig}
\usepackage{epstopdf}
\usetikzlibrary{shapes.geometric, arrows}
\usepackage{fancyhdr}
\usepackage{abstract}
\usepackage{tikz,mathpazo}
\usetikzlibrary{shapes.geometric, arrows}
\usepackage{amssymb}
\usepackage{latexsym}
\usepackage{verbatim}
\usepackage[numbers]{natbib} 
\usepackage{booktabs}

\usepackage{hyperref}
\hypersetup{
    colorlinks=true,
    linkcolor=blue,
    filecolor=magenta,      
    urlcolor=cyan,
    citecolor = cyan,
    }

\allowdisplaybreaks[4]
\newcommand{\inner}[1]{\langle #1 \rangle}
\newcommand{\norm}[1]{\left\Vert #1\right\Vert}
\newcommand{\bb}[1]{\mathbb{#1}}

\newcommand{\ca}[1]{\mathcal{#1}}

\newcommand{\xk}{{x_{k} }}
\newcommand{\yk}{{y_{k} }}

\newcommand{\xkp}{{x_{k+1} }}
\newcommand{\xkm}{{x_{k-1} }}

\newcommand{\ykp}{{y_{k+1} }}
\newcommand{\ykm}{{y_{k-1} }}
\newcommand{\zk}{{z_{k} }}
\newcommand{\zkh}{{z_{k+1/2 }}}
\newcommand{\zkp}{{z_{k+1} }}

\newcommand{\exk}{{\eta_{x,k} }}
\newcommand{\eyk}{{\eta_{y,k} }}
\newcommand{\dxk}{{d_{x,k} }}
\newcommand{\dyk}{{d_{y,k} }}

\newcommand{\Rn}{\mathbb{R}^n}

\newcommand{\Rm}{\mathbb{R}^m}

\newtheorem{theo}{Theorem}[section]
\newtheorem{lem}[theo]{Lemma}
\newtheorem{prop}[theo]{Proposition}

\newtheorem{defin}[theo]{Definition}
\newtheorem{assumpt}[theo]{Assumption}

\theoremstyle{definition}
\newtheorem{rmk}[theo]{Remark}

\usepackage[figuresright]{rotating}
\usepackage{pdflscape}

\newcommand{\Loj}{{{\L}ojasiewicz}\;}
\newcommand{\ys}{\mathcal{Y}^*}
\newcommand{\crit}{\mathrm{crit}}
\newcommand{\cB}{\ca{B}}
\newcommand{\dfx}{\nabla_x f}
\newcommand{\dfy}{\nabla_y f}
\newcommand{\ddfxx}{\nabla^2_{xx} f}
\newcommand{\ddfyy}{\nabla^2_{yy} f}
\newcommand{\ddfxy}{\nabla^2_{xy} f}
\newcommand{\ddfyx}{\nabla^2_{yx} f}
\newcommand{\dhx}{\nabla_x h_\beta}
\newcommand{\dhy}{\nabla_y h_\beta}

\newcommand{\dhyk}{\nabla_y h_{\beta_k}}
\newcommand{\ddhxx}{\nabla^2_{xx} h_\beta}
\newcommand{\ddhyy}{\nabla^2_{yy} h_\beta}
\newcommand{\ddhxy}{\nabla^2_{xy} h_\beta}
\newcommand{\ddhyx}{\nabla^2_{yx} h_\beta}
\newcommand{\qaq}{\quad\text{and}\quad}
\newcommand{\hb}{h_\beta}
\newcommand{\hbk}{h_{\beta_k}}
\newcommand{\sLh}{\mathsf{L}_h}
\newcommand{\eyb}{\bar \eta_y}
\newcommand{\exb}{\bar \eta_x}
\newcommand{\Hk}{H_k}
\newcommand{\Hkp}{H_{k+1}}
\newcommand{\tk}{\tau_k}

\newcommand{\df}{\nabla f}
\newcommand{\dhb}{\nabla h_{\beta}}
\newcommand{\sL}{\mathsf{L}_f}
\newcommand{\sLD}{\mathsf{L}_{f,\mathcal{D}}}

\newcommand{\eykbb}{\eta_{y,k}^{\text{BB1}}}
\newcommand{\exkbb}{\eta_{x,k}^{\text{BB1}}}
\newcommand{\eykbbt}{\eta_{y,k}^{\text{BB2}}}
\newcommand{\exkbbt}{\eta_{x,k}^{\text{BB2}}}
\newcommand{\bk}{\beta_k}
\newcommand{\bkp}{\beta_{k+1}}
\newcommand{\zs}{x, \ys(x)}

\usepackage{caption}
\usepackage{algorithm}
\usepackage{algpseudocode}
\usepackage{longtable}
\usepackage{appendix}
\usepackage{authblk}

\numberwithin{equation}{section}

\usepackage{array}

\title{Line-search and Adaptive Step Sizes for Nonconvex-strongly-concave Minimax Optimization}
\author{Bohao Ma\thanks{School of Data Science, The Chinese University of Hong Kong (Shenzhen), Shenzhen, Guangdong, China 
(\url{bohaoma@link.cuhk.edu.cn}  and \url{xncxy@cuhk.edu.cn}).}, ~
Nachuan Xiao\protect\footnotemark[1], ~
Junyu Zhang\thanks{Department of Industrial Systems Engineering and Management, National University of Singapore, Singapore (\url{junyuz@nus.edu.sg}).}}

\begin{document}
\maketitle
	
\begin{abstract}
    In this paper, we propose a novel reformulation of the smooth nonconvex–strongly-concave (NC–SC) minimax problems that casts the problem as a joint minimization. We show that our reformulation preserves not only first-order stationarity, but also global and local optimality, second-order stationarity, and the Kurdyka-{\L}ojasiewicz (KL) property, of the original NC-SC problem, which is substantially stronger than its nonsmooth counterpart \cite{hu2024minimization} in the literature. With these enhanced structures, we design a versatile parameter-free and nonmonotone line-search framework that does not require evaluating the inner maximization. Under mild conditions, global convergence rates can be obtained, and, with KL property, full sequence convergence with asymptotic rates is also established. In particular, we show our framework is compatible with the gradient descent-ascent (GDA) algorithm. By equipping GDA with Barzilai–Borwein (BB) step sizes and nonmonotone line-search, our method exhibits superior numerical performance against the compared benchmarks.
\end{abstract}

\section{Introduction} \label{sec:intro}
In this paper, we consider the nonconvex-strongly-concave (NC-SC) minimax problem,
\begin{equation}
    \tag{MM}
    \label{Prob_Ori}
    \begin{aligned}
        \min_{x \in \Rn} \max_{y \in \Rm} f(x,y), 
    \end{aligned}
\end{equation}
where $f: \Rn \times \Rm \to \bb{R}$ is locally Lipschitz smooth, possibly nonconvex in $x$, and strongly concave in $y$. The MiniMax  problem \eqref{Prob_Ori} encompasses a wide range of machine learning and statistical applications including generative models \cite{goodfellow2020generative}, adversarial learning \cite{sinha2017certifying}, distributionally robust regression \cite{shafieezadeh2015distributionally}, etc. Throughout this paper, let $\Phi(x):= \max_{y \in \Rm} f(x,y)$ denote the value or primal function for \eqref{Prob_Ori}. We make the following basic assumptions on \eqref{Prob_Ori}.
	\begin{assumpt}%{\bf Blanket assumptions}
		\label{Assumption_1}
		\begin{enumerate} 
			\item The function $f$ is thrice continuously differentiable.
			\item For any $x\in\Rn$, the function $y \mapsto f(x,y)$ is $\mu$-strongly concave, where $\mu > 0$ is independent of $x$.
            \item The value function $\Phi(x)$ is bounded below. That is, $\bar \Phi := \inf_{x\in\Rn} \Phi(x) > -\infty.$
		\end{enumerate}
\end{assumpt}
It is worth mentioning that the minimax problem \eqref{Prob_Ori} is equivalent to the following minimization problem \cite{xu2023unified},
\begin{equation} \label{Prob_value_fcn}
    \min_{x \in \Rn} \Phi(x).
\end{equation}
Additionally, let $\ys(x) := \mathop{\arg\max}_{y \in \Rm} f(x,y)$ denote the maximizer for the inner maximization. The strong concavity of $f(x, \cdot)$ ensures that $\ys$ is singleton-valued, and thus can be regarded as a mapping from $\Rn$ to $\Rm$ in the rest of this paper. 

Based on the reformulation \eqref{Prob_value_fcn}, a wide range of existing first-order algorithms \cite{lin20a,zhang2021complexity,li2021complexity,yang2022nest,bolte2023backtrack,wang2024efficient} for solving the NC–SC minimax problem \eqref{Prob_Ori} adopt a multiple-loop framework. In each outer iteration, these algorithms attempt to minimize $\Phi(x)$ with some first-order method, while Danskin's theorem is used to obtain $\nabla \Phi(x) \approx \nabla_x f\left(x,y_+\right)$ and $y_+ \approx \ys(x)$ is computed by an inner loop that approximately maximizes $f(x,\cdot)$. Combined with the (inexact) proximal point algorithm (PPA) \cite{rockafellar1976monotone}, in each strongly-convex-strongly-concave subproblem, \cite{lin20a,zhang2021complexity,wang2024efficient} employ the  Nesterov acceleration \cite{nesterov1983method} for the inner maximization sub-subproblems, demonstrating (near-)optimal complexities that (almost) match the information theoretic lower bounds \cite{lin20a,zhang2021complexity,li2021complexity,wang2024efficient}. Yet, in practice, these triple-loop frameworks are difficult in tuning parameters, balancing the computational costs among multiple loops, as well as keeping the stability of the algorithm. Despite the theoretical (near-)optimality, such a drawback significantly limits the practical efficiency of these algorithms. 

Compared to the multiple-loop PPA type algorithms, a more practically implementable class of methods are the {\it single-loop} descent–ascent methods \cite{lin2020gradient,mahdavinia2022tight,li2023tiada,huang23a,xu2023unified,xu2024stochasticgdamethodbacktracking,zhang2024agdaproximalalternatinggradient,lin2025two,li2023tiada,huang23a}. In each iteration, these methods  update the $x$-variable by taking an approximated gradient descent step to $\Phi(x)$, while the $y$-variable is updated to track $\ys(x)$ by taking a single ascent step or other specifically designed schemes for the inner maximization. %subproblem.  
For example, \cite{lin2020gradient,lin2025two} propose a gradient descent-ascent (GDA) method, by setting the ratio between the $y$- and $x$-step sizes to be sufficiently large, yet still finite, and provable convergence to $\varepsilon$-stationary points of $\Phi(x)$ can be established. Besides GDA,  optimistic gradient descent-ascent (OGDA) method and the extra-gradient (EG) method have also be studied under the NC-SC setting \cite{mahdavinia2022tight}.

However, a significant limitation of the existing works is their heavy reliance on prior knowledge of the problem-specific global parameters (e.g., $\mu$ and $\sL$), which in turn allow them to carefully design the step sizes, making it difficult to tune parameters in practice. For minimization problems, there are mature line-search tools \cite{Bertsekas2016NonlinearProgramming} to resolve this issue. In addition to classic monotone line-search results, highly efficient adaptive step sizes such as Barzilai–Borwein (BB) \cite{raydan1997barzilai} are also provably  convergent under the nonmonotone line-search frameworks \cite{zhang2004nonmonotone}. Moreover, under the Kurdyka-\Loj (KL) property \cite{lojasiewicz1965ensembles,kur98}, line-search methods enjoy strong full sequence convergence and fast local rates, see, e.g., \cite{qian2023convergence,qian2025convergence}.
Yet, unfortunately, for problem \eqref{Prob_Ori}, the nested minimax structure precludes the direct use of $f$ as the potential function in line-search. To our best knowledge, the only existing line-search algorithm for \eqref{Prob_Ori} is \cite{bolte2023backtrack}, which applies the Armijo rule to the value function $\Phi$. In this method, each line-search iteration involves repeated evaluation of $\Phi$, which in turn requires repeatedly solving the inner maximization problem to sufficiently high accuracy, making the algorithm computationally inefficient for practical problems. Therefore, we are driven to the following question:
\begin{quote}
    Can we develop an efficient and easy-to-implement line-search framework for the %minimax 
    problem \eqref{Prob_Ori} with guaranteed global convergence? 
\end{quote}
In this paper, we provide an affirmative answer to this question by considering the following smooth Regularized Minimization reformulation \eqref{Prob_Min} for problem \eqref{Prob_Ori}: 
\begin{equation}
\tag{RM}\label{Prob_Min}
\min_{x\in\Rn, y\in\Rm} \quad \hb(x,y)
:= f(x,y) + \frac{\beta}{2} \norm{\nabla_y f(x,y)}^2.
\end{equation}
Our reformulation \eqref{Prob_Min} is partly inspired by the Partial Forward-Backward Envelope (PFBE) reformulation in \cite{hu2024minimization}, yet \eqref{Prob_Min} preserves much stronger structural information of \eqref{Prob_Ori} than the general PFBE. In particular, under a mild condition that regularization parameter $\beta>\mu^{-1}$, we prove:
\begin{enumerate}
    \item Problems \eqref{Prob_Min} and \eqref{Prob_Ori} share the same set of first- and second-order stationary points.
    \item Problems \eqref{Prob_Min} and \eqref{Prob_Ori} share the same set global and local optimal points. 
    \item If $\Phi$ satisfies the KL property, then $h_\beta$ also preserves the KL property with the same KL exponent. 
\end{enumerate}
Most importantly, unlike the value function $\Phi$, the evaluation of $\hb$ does not rely on any inner maximization subroutines. Therefore, $\hb$ is directly applicable to the line-search techniques as a natural metric to measure the quality of a trial point, which is the first major contribution of this work. 

As for our second contribution, we note that the gradient $\dhb$ involves the (partial) Hessian computation of $\nabla_{yy}^2 f$. Therefore a naive implementation of GD with line-search for \eqref{Prob_Min} will lead to  potential computational overheads, and the resulting method significantly diverges from the widely accepted GDA type schemes. To resolve this issue, we propose a general nonmonotone line-search framework (Algorithm \ref{alg:line-search-framework}) for the NC–SC minimax problem \eqref{Prob_Ori}. Under mild conditions, we prove the global convergence and an $O(\varepsilon^{-2})$-iteration complexity for our method. Given the KL property of \eqref{Prob_Ori} (defined for $\Phi$), full sequence convergence and a fast local rate can be further established.

Moreover, when prior knowledge of the value of $\mu$ is not accessible, we also propose a fully parameter-free GDA method with line-search (Algorithm \ref{alg:parameter-free-GDA}), and we show that it still shares the same convergence guarantees as Algorithm \ref{alg:line-search-framework}.  

Finally, motivated by the %Barzilai-Borwein 
BB step size for minimization problems, we propose a BB step size rule for NC-SC minimax optimization problem \eqref{Prob_Ori}, with provable global convergence. Preliminary numerical experiments on robust nonlinear regression problems demonstrates that when equipped with the proposed BB stepsize rule, both Algorithm \ref{alg:line-search-framework} and Algorithm \ref{alg:parameter-free-GDA} demonstrate significant superiority over the two-timescale GDA method \cite{lin2020gradient,lin2025two} and the L-BFGS method directly applied to the minimization reformulation \eqref{Prob_Min}. 
\vspace{0.2cm}

\noindent\textbf{Additional related works on KL. }  In this paragraph, we supplement the recent works that leverage the KL or global Polyak–Łojasiewicz (PL) properties on \eqref{Prob_Ori} to analyze first-order methods  \cite{nouiehed2019solving,zheng2022doubly,zheng2023universal,zheng2025doubly,li2025nonsmooth,lu2025first,lu2025first2,cohen2025alternating}. These results typically assume a one-sided KL or PL condition, imposed either on the minimization part, the maximization part, or both, and use it as a regularity condition to establish complexity results. Furthermore, \cite{cohen2025alternating} recently establishes KL-based iterate convergence of (proximal) GDA methods with constant step sizes under the NC-SC setting, by constructing a Lyapunov function that couples the value function and the maximizer mapping. In contrast, our reformulation \eqref{Prob_Min} provides a much cleaner single-objective Lyapunov function, allowing us to derive convergence of the full iteration sequence and fast local rates with more natural analysis.\vspace{0.2cm}

\noindent\textbf{Organization. } 
The rest of this paper is organized as follows. Section \ref{sec:preliminaries} introduces the definitions and preliminary results required throughout the paper. In section \ref{sec:linesearch-framework}, we establish the equivalence between the minimax problem \eqref{Prob_Ori} and its minimization reformulation \eqref{Prob_Min} with respect to global and local solutions, first-order minimax (stationary) points, second-order minimax (stationary) points, and the KL property. We then present our line-search framework (Algorithm \ref{alg:line-search-framework}), and establish its convergence guarantees. Section \ref{sec:applications} illustrates the application of this framework by developing a parameter-free GDA method equipped with line-search and adaptive BB-type step sizes (Algorithm \ref{alg:parameter-free-GDA}). Section \ref{sec:numerical experiments} reports numerical experiments demonstrating the empirical efficiency of the proposed approach. We conclude in the final section. \vspace{0.2cm}

\noindent\textbf{Notations. }  
For $v,w \in \bb{R}^d$, we denote the Euclidean inner product as $\inner{v, w} := \sum_{i=1}^d v_i w_i$, the Euclidean norm as $\|v\| := \sqrt {\inner{v, v}}$, and the ball $\cB_\delta(v) := \{w \in \bb{R}^d: \|w-v\| \leq \delta\}$. For a function $F:\bb{R}^d \to \bb{R}$ and a real number $\alpha \in \bb{R}$, we denote the sublevel set of $F$ by $[F(v) \leq \alpha] := \{v \in \bb{R}^d: F(v) \leq \alpha\}$. The $d$-dimensional identity matrix is denoted as $I_d \in \bb{R}^{d\times d}$. The transposed Jacobian of a mapping $F: \bb{R}^{d_1} \to \bb{R}^{d_2}$ is denoted as $\nabla F(v) \in \bb{R}^{d_1 \times d_2}$. That is, let $F_i : \bb{R}^{d_1} \to \bb{R}$ denote the $i$-th coordinate of the mapping $F$ and then $\nabla F$ is defined as $\nabla F(v):= \left[\nabla F_1(v), ..., \nabla F_{d_2}(v) \right],$
where each $\nabla F_i(v)$ is a column vector in $\bb{R}^{d_1}$. Moreover, for $w \in \bb{R}^{d_2}$, $\nabla^2 F(x): y \mapsto \nabla^2 F(v)[w]$ denotes the second-order (directional) derivative of the mapping $F$, which can be regarded as a linear mapping from $\bb{R}^{d_2}$ to $\bb{R}^{d_1 \times d_1}$ satisfying
$\nabla^2 F(v)[w] = \sum_{i=1}^{d_2} w_i \nabla^2 F_i(v).$
For $f: \Rn \times \Rm \to \bb{R}$, we define its partial gradient, Hessian, and third (directional) derivative as $\nabla_x f(x,y) := \nabla g_{y}(x) \in \Rn$, $\nabla_{yx}^2 f(x,y) := \nabla g'_{x}(y) \in \bb{R}^{m \times n}$, and $\nabla_{xyx}^3 f(x,y)[z] := \sum_{i=1}^n z_i \nabla^2_{xy} g_{x,i}'(x,y) \in \bb{R}^{n \times m}$, respectively, with $g_{y}(x) = f(x, y)$, $g'_{x}(y) = \nabla_x f(x,y)$, and $g_{x,i}' = (\nabla_x f)_i$.
All the other partial gradients, Hessians, and third derivatives are defined in a similar manner. Given a bounded set $\mathcal{D}$, we will denote the Lipschitz modulus of $\nabla f$ over $\mathcal{D}$ as $\sLD$ throughout. Finally, for simplicity, we denote $\zk := (\xk, \yk)$,  $\zkh := (\xk, \ykp)$, and $\zkp := (\xkp, \ykp)$ in the proof of theoretical results in Section \ref{sec:linesearch-framework} and \ref{sec:applications}. 

\section{Preliminaries} \label{sec:preliminaries}

\subsection{Optimality Conditions and the Kurdyka-{{\L}ojasiewicz} (KL) Property}
\label{subsec:opt_cond_KL}

In this section, we present the preliminary concepts about the optimality conditions for the minimax problem \eqref{Prob_Ori} and the minimization problem \eqref{Prob_Min}, as well as the concept of the KL %Kurdyka-\Loj 
property. We begin with the definitions on the global and local optimal (minimax) points of problem \eqref{Prob_Ori}. %and \eqref{Prob_Min}.

\begin{defin}
    \label{def:global-local-minimax-point-ori}
    We say that $(x^*, y^*)$ is a global minimax point of \eqref{Prob_Ori} if
    \begin{equation*}
        %\label{eq:def-global-minimax-ori}
        x^* \in \mathop{\arg \min}_{x \in \Rn} \Phi(x) \quad \text{and} \quad 
        y^* \in \ys(x^*).
    \end{equation*}
    And we say that $(x^*, y^*)$ is a local minimax point of \eqref{Prob_Ori} if there exists $\delta > 0$ s.t.
    \begin{equation*}
        %\label{eq:def-local-minimax-ori}
        x^* \in \mathop{\arg\min}_{x \in \Rn \cap \cB_\delta(x^*)} \Phi(x) \quad \text{and} \quad 
        y^* \in \ys(x^*).
    \end{equation*} 
\end{defin}

For global convergence analysis, we also introduce the definitions on the first-order minimax point of \eqref{Prob_Ori}, which is defined as the first-order stationary point of the (possibly nonconvex) value function $\Phi$.

\begin{defin}
    \label{def:opt-cond-ori}
    We say that $(x^*, y^*)$ is a first-order minimax point of \eqref{Prob_Ori} if $\nabla \Phi(x^*) = 0$ and $y^* \in \ys(x^*)$. Or equivalently, $\dfx(x^*, y^*) = 0$ and $\dfy(x^*, y^*) = 0.$     Furthermore, we say that $(x^*, y^*)$ is an $\varepsilon$-first-order minimax point of \eqref{Prob_Ori} if $\|\dfx(x^*, y^*)\| \leq \varepsilon$ and $\|\dfy(x^*, y^*)\| \leq \varepsilon.$
\end{defin}

Similar to the first-order optimality conditions, the second-order optimality conditions for \eqref{Prob_Ori} can also be defined %similarly 
through that of the value function $\Phi$. 

\begin{defin}
    \label{second-order-opt-cond-ori}
    Let $(x^*, y^*)$ be a first-order minimax point of \eqref{Prob_Ori}. If additionally $\nabla^2\Phi(x^*) \succeq 0$, then we say that $(x^*, y^*)$ satisfies the second-order necessary condition (SONC) of \eqref{Prob_Ori}. If $\nabla^2\Phi(x^*) \succ 0$, then we say that $(x^*, y^*)$ satisfies the second-order sufficient condition (SOSC).
\end{defin}

For the regularized minimization problem \eqref{Prob_Min}, the local and global optimal points, as well as the first- and second-order stationary points, follow the conventions of nonconvex optimization for the joint minimization over $(x,y)$, which are stated below for presentation completeness.  
\begin{defin}
    \label{def:global-local-optimal-min}
    We say that $(x^*, y^*)$ is a global minimizer of \eqref{Prob_Min} if
    \begin{equation*}
        %\label{eq:def-global-minimax-ori}
        (x^*,y^*) \in \mathop{\arg \min}_{x \in \Rn, y\in\Rm} \hb(x,y).
    \end{equation*}
    And we say that $(x^*, y^*)$ is a local minimizer of \eqref{Prob_Min} if there exists $\delta > 0$ such that
    \begin{equation*}
        %\label{eq:def-local-minimax-ori}
        (x^*,y^*) \in \mathop{\arg \min}_{ (x,y)\in \cB_\delta(x^*,y^*)} \hb(x,y).
    \end{equation*}
\end{defin}

\begin{defin}
    \label{def:opt-cond-min}
    A point $(x^*, y^*)$ is called a first-order stationary point of the problem \eqref{Prob_Min} if $\dhx(x^*, y^*) = 0$ and $\dhy(x^*, y^*) = 0.$
    It is called an $\varepsilon$-first-order stationary point of \eqref{Prob_Min} if $\|\dhx(x^*, y^*)\| \leq \varepsilon$ and $\|\dhy(x^*, y^*)\| \leq \varepsilon.$
\end{defin}

\begin{defin}
    \label{second-order-opt-cond-min}
    Let $(x^*, y^*)$ be a first-order stationary point of \eqref{Prob_Min}. If additionally $\nabla^2 \hb (x^*, y^*) \succeq 0$, then we say that $(x^*, y^*)$ satisfies the SONC of \eqref{Prob_Min}.
    If $\nabla^2 \hb (x^*, y^*) \succ 0$,  then we say that $(x^*, y^*)$ satisfies the SOSC of \eqref{Prob_Min}.
\end{defin}

To establish iterate convergence and local convergence rates for our line-search framework, we formally introduce the KL property \cite{lojasiewicz1965ensembles,kur98}, a mild assumption on the local geometry of the objective function. It holds for all (smooth) subanalytic and semialgebraic functions and plays a key role in establishing iterate convergence and convergence rates of first-order methods, especially under the nonconvex setting \cite{lojasiewicz1965ensembles,kur98,absil2005convergence,AttBol09,AttBolRedSou10,AttBolSva13,BolSabTeb14}. Furthermore, a broad class of practical problems satisfy this property, see, e.g., \cite[Section 4]{AttBolRedSou10} and \cite[Section 5]{BolSabTeb14}.

\begin{defin}[KL property]
    \label{def:Loj-inequality}
    A differentiable function $F: \bb{R}^d \to \bb{R}$ is said to satisfy the KL property at a point $\bar x$ if there exists a neighborhood $U(\bar x)$ of $\bar x$ such that
    \[\hspace{-2ex}\|\nabla{F}(x)\|\geq C_F|F(x)-F(\bar x)|^\theta, \quad \forall~x \in U(\bar x),\]
    where $C_F = C_F(\bar x) >0$ and $\theta = \theta(\bar x) \in[\frac12,1)$ are constants depending on the point $\bar{x}$. In particular,  $\theta$ is called the %{\L}ojasiewicz 
    KL exponent of $F$ at $\bar x$.
\end{defin}

It is worth noting that the KL property holds for more general desingularization function (see, e.g., \cite{AttBol09,AttBolSva13}) and Definition \ref{def:Loj-inequality} is sometimes called the {\L}ojasiewicz property in the literature. We will stick to Definition \ref{def:Loj-inequality} as it allows an explicit asymptotic convergence rates. 

\subsection{Preliminary Lemmas}
\label{subsec:preliminary-lemmas}
As a preparation for later analysis, we will present a few key preliminary lemmas in this subsection. We begin with a lemma on the smoothness of on the functions $\ys$ and $\Phi$, which directly follows from \cite[Theorem 10.58]{rockafellar1998variational}. Therefore, we omit the proof for Lemma \ref{lem:properties-of-value-fcn}. 
\begin{lem}
    \label{lem:properties-of-value-fcn}
    Given Assumption \ref{Assumption_1}, the optimal-value mapping $\ys$ is continuously differentiable with
    \[
    \nabla \ys(x) = -\ddfxy(\zs)[\ddfyy(\zs)]^{-1}.
    \]
    The value function $\Phi$ is twice continuously differentiable with $\nabla \Phi(x) = \nabla_x f(\zs)$ and
    \begin{equation}\label{eq:grad-hessian-of-Phi}  \nabla^2\Phi(x) ={} \ddfxx(\zs) - \ddfxy(\zs)[\ddfyy(\zs)]^{-1}\ddfyx(\zs).
    \end{equation}
\end{lem}
\noindent Next, we present the smoothness and lower boundedness of the regularized objective function $\hb$.
\begin{lem}
    \label{lem:properties-of-h}
    Suppose Assumption \ref{Assumption_1} holds and $\beta > 0$. Then $\hb$ is twice continuously differentiable with
        \begin{align*}
            \dhx(x, y) &= \dfx(x, y) + \beta \ddfxy(x, y)\dfy(x, y),\\
            \dhy(x, y) &=[I_m+\beta\ddfyy(x, y)] \dfy(x, y).
        \end{align*}
        If in addition that the regularization parameter $\beta$ is large enough s.t. $\beta > \mu^{-1}$,  then %we have 
        \[\inf_{x\in \Rn, y \in \Rm} h_{\beta}(x, y) = \inf_{x \in \Rn} \Phi(x) =: \bar \Phi > -\infty.\] 
\end{lem}
\begin{proof}
    The smoothness of $\hb$ and its gradient formulas follow from direct computations and are thus omitted. For the lower boundedness, let us 
    fix an arbitrary $\beta > \mu^{-1}$. Because $f(x,\cdot)$ is $\mu$-strongly concave, it holds that
    \begin{align*}
        \hb(x,y) &= f(x,y) + \frac{\beta}{2} \|\dfy(x,y)\|^2
        \overset{(i)}{\geq} f(x,y) + \beta\mu [f(x,\ys(x)) - f(x,y)]\\
        &\overset{(ii)}{=} \beta\mu \Phi(x) + (1- \beta\mu) f(x,y) \overset{(iii)}{\geq} \Phi(x), %\geq \bar\Phi, 
        \quad \forall (x,y) \in \Rn \times \Rm,
    \end{align*}
    where (i) is by the standard strong concavity property: $f(x,\ys(x))-f(x,y) \leq \frac{1}{2\mu} \|\dfy(x,y)\|^2$;
    (ii) follows from the definition of $\Phi$; and (iii) is due to $1-\beta\mu<0$ when $\beta>\mu^{-1}$ and $\Phi(x) \geq f(x,y)$, for all $x, y$. Therefore, \[\inf_{x\in \Rn, y \in \Rn} h_{\beta}(x, y) \geq \inf_{x \in \Rn} \Phi(x).\] For the reverse direction, since $\dfy(x,\ys(x)) = 0$, it holds that $\hb(x,\ys(x)) = \Phi(x) $ for all $x \in \Rn$, which further implies
    \[
    \inf_{x \in \Rn} \Phi(x) = \inf_{x \in \Rn} \hb(x, \ys(x)) \geq \inf_{x \in \Rn, y \in \Rm} \hb(x,y).
    \]
    Combining the above inequalities completes the proof of this lemma.  
\end{proof}
\noindent It is worth mentioning that, although the gradient of the regularized objective function  $\hb$ involves the Hessian-vector product operations, it is mostly used for theoretical analysis. Our line-search framework allows us to directly use the gradient of the original objective function $f$ as line-search directions. 

Next, for any $\beta>0$, we can compute the Hessian of $\hb$ as a straightforward practice of chain-rule:  
\begin{align}
    \nabla^2 h(x,y) &= 
        \begin{pmatrix}
            \ddhxx(x, y) & \ddhxy(x, y)\\
            \ddhxy(x, y)^\top & \ddhyy(x, y)
        \end{pmatrix} \label{eq:hessian-of-h},
    \end{align}
    where each block of the above Hessian takes the form
    \begin{equation}
    \begin{aligned}
    \label{eq:formula-for-ddhyy}
            \ddhxx(x, y) &= \ddfxx(x, y) + \beta\{\ddfxy(x, y)\ddfyx(x, y)
            +\nabla^3_{xxy} f(x, y)[\dfy(x, y)]\},\\
            \ddhxy(x, y) &= \ddfxy(x, y) + \beta\{\ddfxy(x, y)\ddfyy(x, y) + \nabla^3_{xyy} f(x, y)[\dfy(x, y)]\}, \\
            \nabla^2_{yy} \hb(x, y)
            &= \ddfyy(x, y) + \beta\{[\ddfyy(x, y)]^2 + \nabla^3_{yyy} f(x, y)[\dfy(x, y)]\}. 
    \end{aligned}
    \end{equation}
\iffalse
\begin{lem}
    \label{lem:second-order-diff-h}
    Suppose Assumption \ref{Assumption_1} holds and $\beta > 0$. Then $\hb$ is twice continuously differentiable with
    \begin{align}
        \nabla^2 h(x,y) &= 
            \begin{pmatrix}
                \ddhxx(x, y) & \ddhxy(x, y)\\
                \ddhyx(x, y) & \ddhyy(x, y)
            \end{pmatrix} \label{eq:hessian-of-h},
        \end{align}
        where 
        \begin{equation}
            \label{eq:formula-for-ddhyy}
            \begin{aligned}
            \ddhxx(x, y) &= \ddfxx(x, y) + \beta\{\ddfxy(x, y)\ddfyx(x, y)
            +\nabla^3_{xxy} f(x, y)[\dfy(x, y)]\},\\
            \ddhxy(x, y) &= \ddfxy(x, y) + \beta\{\ddfxy(x, y)\ddfyy(x, y) + \nabla^3_{xyy} f(x, y)[\dfy(x, y)]\} = \ddhyx(x, y)^\top,\\
            \nabla^2_{yy} \hb(x, y)
            &= \ddfyy(x, y) + \beta\{[\ddfyy(x, y)]^2 + \nabla^3_{yyy} f(x, y)[\dfy(x, y)]\}. 
            \end{aligned}
        \end{equation}
\end{lem}
\fi
In particular, the term $\dfy(x, y)=0$ when $y = \ys(x)$, which eliminates the third-order derivatives in \eqref{eq:formula-for-ddhyy}, giving a much cleaner formula for $\nabla^2h$.  

As we emphasized in the introduction, the regularized function $\hb$ inherits a lot of structural information from the original problem. In the following lemma, we show that the coercivity of the value function $\Phi$ can be inherited by $\hb$.
\begin{lem}
    \label{lem:coercivity-of-h}
    Suppose Assumption \ref{Assumption_1} holds and the value function $\Phi$ is coercive. If $\beta > \mu^{-1}$, then the function $\hb$ is also coercive. 
    Moreover, if $\beta \in (0, \mu^{-1}]$, then the set $[\hb(x,y) \leq \alpha] \cap [\|\nabla_y f(x,y)\| \leq G]$ is bounded for all $G \geq 0$ and $\alpha \in \bb{R}$.
\end{lem}
\begin{proof}
    Since $f(x, \cdot)$ is $\mu$-strongly concave, it holds that
    \begin{equation}
    \label{eq:h-coer-inter1}
        \hb(x,y) - \Phi(x) =  f(x,y) - f(x, \ys(x)) + \frac{\beta}{2}\|\dfy(x,y)\|^2  \geq \Big(\frac\beta 2 - \frac{1}{2\mu}\Big) \|\dfy(x,y)\|^2.
    \end{equation}
    Fix any $\beta > \mu^{-1}$ and $\alpha \in \bb{R}$. The level set satisfies that 
    \begin{align*}
         [\hb(x,y) \leq \alpha] &\subseteq \left[\Phi(x) + \Big(\frac\beta 2 - \frac{1}{2\mu}\Big) \|\dfy(x,y)\|^2\leq \alpha\right]\\
         &\subseteq [\Phi(x) \leq \alpha] \cap \left[\|\dfy(x,y)\|^2 \leq \frac{2(\alpha - \bar \Phi)}{\beta - \mu^{-1}}\right]\\
         &\subseteq [\Phi(x) \leq \alpha] \cap \left[\|y - \ys(x)\|^2 \leq \frac{2(\alpha - \bar \Phi)}{\mu(\beta\mu - 1)}\right],
    \end{align*}
    where the last inclusion follows from $\|\nabla_y f(x,y)\| \geq \mu \|y-\ys(x)\|$. Since $\Phi$ is coercive and $\ys$ is continuous by Lemma \ref{lem:properties-of-value-fcn}, the last intersection yields a bounded set and thus $\hb$ is coercive. 
    
    Now, for any $0<\beta \leq \mu^{-1}$, $\alpha \in \bb{R}$, and $G \geq 0$, observe that \eqref{eq:h-coer-inter1} implies 
    \begin{align*}
         &[\hb(x,y) \leq \alpha] \cap [\|\nabla_y f(x,y)\| \leq G] \subseteq \left[\Phi(x) \leq \alpha + \frac{G^2(1-\beta\mu)}{2\mu}\right] \cap [\|\nabla_y f(x,y)\| \leq G]\\
         &\hspace{4.8cm}\subseteq \left[\Phi(x) \leq \alpha + \frac{G^2(1-\beta\mu)}{2\mu}\right] \cap \left[\|y - \ys(x)\| \leq \frac{G}{\mu}\right],
    \end{align*}
    which is a bounded set. 
\end{proof}

In the following, we prove some technical lemmas for later sections. It is worth noting that in Lemma \ref{lem:various-gradient-bounds0} and \ref{lem:various-gradient-bounds1}, we only require $\beta$ to be positive.
\begin{lem}
    \label{lem:various-gradient-bounds0}
    Suppose Assumption \ref{Assumption_1} holds and $\beta > 0$. Let $\mathcal{D}$ be a compact set and let $\df$ be $\sLD$-Lipschitz continuous in $\mathcal{D}$. 
    Then, for any $(x, y) \in \mathcal{D}$, we have 
    \begin{align}
        \|\dhy(x,y)\|  &\leq (1 +\beta \sLD) \|\dfy(x,y)\|, \nonumber \\ %\label{eq:upper-bd-dhy}\\
        \|\dhx(x,y)\| &\leq \|\dfx(x,y)\| + \beta \sLD \|\dfy (x,y)\|. \nonumber%\label{eq:upper-bd-dhx}
    \end{align}
\end{lem}
The proof of this lemma is a straightforward application of Lemma \ref{lem:properties-of-h}, and is thus omitted.  Next, we present two preliminary lemmas that relate the gradient of the regularized objective function $\hb$ and the gradient of the original objective $f$.  
\iffalse
\begin{proof}
    Fix $(x,y) \in \mathcal{D}$ and $\beta > 0$. Invoking the triangular inequality and $\sLD$-smoothness of $f$ yields
    \begin{align*}
    \|\dhy(x,y)\| &= \|[I_m + \beta \ddfyy(x,y)] \dfy(x,y)\| \leq (1+\beta \sLD) \|\dfy(x,y)\|, \\
    \|\dhx(x,y)\| &= \|\dfx(x,y) + \beta\ddfyy(x,y)\dfy(x,y)\| \leq \|\dfx(x,y)\| + \beta\sLD\|\dfy(x,y)\|.
    \end{align*}
    Hence we complete the proof. 
\end{proof}
\fi

\begin{lem}
    \label{lem:various-gradient-bounds1}
   Under the same conditions of Lemma \ref{lem:various-gradient-bounds0}, for any $(x, y) \in \mathcal{D}$, it holds that
    \begin{align} 
        2\inner{\dhx(x, y), -\dfx(x, y)} + \|\nabla_x f(x, y)\|^2
        &\leq \beta^2\sLD^2 \|\nabla_yf(x,y)\|^2; \nonumber \\ %\label{eq:search-dir-bd-gda-x}\\
        (1-C^{-1}) \|\nabla_x f(x,y)\|^2 - \beta^2  (C - 1) \sLD^2 \|\dfy(x,y)\|^2 &\leq \|\dhx(x,y)\|^2, \quad \forall C > 1. \nonumber %\label{eq:lower-bd-gradient-hx-generalC}
    \end{align}
\end{lem}

\begin{proof}
    Invoking the Young's inequality, that is, $C \|v\|^2 +C^{-1} \|w\|^2 \ge 2\inner{v,w}$, $\forall C > 0$, and Lipschitz continuity of $\df$ over $\mathcal{D}$ yields
    \begin{equation}
        \label{eq:upper-lower-bd-inter0}
        \begin{aligned}
        &2\inner{\beta\ddfxy(x, y)\dfy(x, y), -\dfx(x, y)}\\
        \leq& C\beta^2 \|\ddfxy(x,y)\dfy(x, y)\|^2 + C^{-1} \|\dfx(x, y)\|^2, \\
        \leq& {C\beta^2\sLD^2} \|\dfy(x, y)\|^2 + C^{-1} \|\dfx(x, y)\|^2. 
    \end{aligned}
    \end{equation}
    This implies that for any $C > 0$,
    \begin{align}
        2\inner{\dhx(x, y), -\dfx(x, y)}
        &= 2\inner{\dfx(x, y) + \beta\ddfxy(x, y)\dfy(x, y), -\dfx(x, y)} \nonumber\\
        &\leq -(2-C^{-1}) \|\dfx(x, y)\|^2 + {C\beta^2\sLD^2} \|\dfy(x, y)\|^2. \nonumber %\label{eq:upper-lower-bd-inter2}
    \end{align}
    Setting $C = 1$ yields the first inequality. For the second inequality, observe that
    \begin{align*}
    &\|\nabla_x \hb(x,y)\|^2 
    = \|\nabla_x f(x,y) + \beta \nabla^2_{xy} f(x,y) \nabla_y f(x,y)\|^2 \nonumber\\
    =& \|\nabla_x f(x,y)\|^2 + \beta^2\|\nabla^2_{xy} f(x,y) \nabla_y f(x,y)\|^2 + 2\inner{\nabla_x f(x,y), \nabla^2_{xy} f(x,y) \nabla_y f(x,y)} \nonumber\\
    \geq& (1-C^{-1}) \|\nabla_x f(x,y)\|^2 - \beta^2  (C - 1) \|\nabla^2_{xy} f(x,y)\nabla_y f(x,y)\|^2 \nonumber\\
    \geq& (1-C^{-1}) \|\nabla_x f(x,y)\|^2 - \beta^2  (C - 1) \sLD^2 \|\dfy(x,y)\|^2, \quad \forall C > 1, 
    \nonumber
    \end{align*}
    where the first inequality follows from the first inequality in $\eqref{eq:upper-lower-bd-inter0}$ and 
    the last line is due to $C>1$ and Lipschitz continuity of $\nabla f$ over $\mathcal{D}$. This completes the proof.
\end{proof}

For the last technical lemma of this subsection, we require $\beta > \mu^{-1}$.
\begin{lem}
    \label{lem:various-gradient-bounds2}
     Suppose Assumption \ref{Assumption_1} holds and $\beta > \mu^{-1}$. It holds that
        \begin{align}
            \inner{\dhy(x, y), \dfy(x, y)}
            &\leq -(\beta \mu - 1) \|\nabla_y f(x, y)\|^2; \nonumber \\ %\label{eq:search-dir-bd-gda-y}\\
            (\beta \mu - 1) \|\dfy(x,y)\| &\leq \|\dhy(x,y)\|.\nonumber %\label{eq:lower-bd-gradient-hy}
        \end{align}
\end{lem}
\begin{proof}
    Let us fix an arbitrary $\beta > \mu^{-1}$. Since $\ddfyy(x,y) \preceq -\mu I_m$, it holds that
    \begin{equation*}
        %\label{eq:Id+Hess-bound}
        I_m + \beta \ddfyy(x,y) \preceq -(\beta\mu-1) I_m \prec 0.
    \end{equation*}
    This implies that
    \begin{equation*}
    \begin{aligned}
    \inner{\dhy(x, y), \dfy(x, y)} &= \inner{[I_m + \beta \ddfyy(x,y)] \dfy(x, y), \dfy(x, y)}\\ &\leq -(\beta\mu-1) \|\dfy(x, y)\|^2,
    \end{aligned}
    \end{equation*}
    and that
    \begin{equation*}
    \|\nabla_y \hb(x,y)\| = \|[I_m+\beta\nabla^2_{yy} f(x,y)]\nabla_y f(x,y)\| \geq (\beta \mu - 1) \|\nabla_y f(x,y)\|,
    \end{equation*}
    as desired.
\end{proof}

\section{A Line-search Framework for Minimax Problems}
\label{sec:linesearch-framework}

\subsection{Equivalence Between \eqref{Prob_Ori} and \eqref{Prob_Min}}\label{subsec:opt-cond-equiv}
In this subsection, we establish the equivalence between the minimax problem \eqref{Prob_Ori} and the regularized minimization problem \eqref{Prob_Min} when $\beta > \mu^{-1}$. First, let us begin with the equivalence between the $\varepsilon$-first-order minimax (stationary) points of the two problems.

\begin{prop}
    \label{prop:epsilon_FOSP1}
    Suppose Assumption \ref{Assumption_1} holds and $\beta > \mu^{-1}$. Let $(x^*, y^*)$ be an $\varepsilon$-first-order minimax point of \eqref{Prob_Ori}, $\varepsilon\geq0$. Let $\mathcal{D}$ be a compact set that contains $(x^*, y^*)$ and $f$ be $\sLD$-smooth in $\mathcal{D}$. Then $(x^*, y^*)$ is an $(1+\beta\sLD)\varepsilon$-first-order stationary point of \eqref{Prob_Min}.
\end{prop}
\begin{proof}
    By definition of $\varepsilon$-first-order minimax point, we have $\|\dfy(x^*, y^*)\| \leq \varepsilon$, $\|\dfx(x^*, y^*)\| \le \varepsilon.$ Then directly applying Lemma \ref{lem:various-gradient-bounds0} proves the proposition. 
    %which, by \eqref{eq:upper-bd-dhy} and \eqref{eq:upper-bd-dhx}, implies that
    %\begin{align*}
    %    \|\dhy(x^*, y^*)\| \leq (1+\beta\sLD) \varepsilon \qaq
    %    \|\dhx(x^*, y^*)\| \leq (1+\beta\sLD) \varepsilon,
    %\end{align*} 
    %as desired.
\end{proof}

\begin{prop}
    \label{prop:epsilon_FOSP2}
    Suppose Assumption \ref{Assumption_1} holds and $\beta > \mu^{-1}$. Let $(x^*, y^*)$ be an $\varepsilon$-first-order stationary point of \eqref{Prob_Min}, $\epsilon\geq0$. Let $\mathcal{D}$ be a compact set that contains $(x^*, y^*)$ and  $f$ be $\sLD$-smooth in $\mathcal{D}$. Then $(x^*, y^*)$ is a $\sqrt{2[1+(\beta\sLD)^2/(\beta\mu-1)^2]}\cdot\varepsilon$-first-order minimax point of \eqref{Prob_Ori}.
\end{prop}
\begin{proof}
    By definition of $\varepsilon$-first-order stationary point, we have $\|\dhy(x^*, y^*)\| \leq \varepsilon$, $\|\dhx(x^*, y^*)\| \leq \varepsilon.$
    By the second inequality of Lemma \ref{lem:various-gradient-bounds2}, it holds that
    \begin{gather*}
        \|\dfy(x^*, y^*)\| \leq \varepsilon / (\beta \mu - 1).
    \end{gather*}
    Together with the second inequality of Lemma \ref{lem:various-gradient-bounds1} ($C = 2$), we further obtain
    \begin{align*}
        \|\dfx(x^*, y^*)\|^2 &\leq 2\|\dhx(x^*, y^*)\|^2 + 2\beta^2\sLD^2\|\dfy(x^*, y^*)\|^2\\
        &\leq 2[1+\beta^2\sLD^2/(\beta \mu - 1)^2] \varepsilon^2.
    \end{align*}
    Combining the above two inequalities, we complete the proof of the proposition. 
\end{proof}

In particular, if one adopts the common assumption that %the function 
$f$ is globally $L$-smooth, then one can replace the local Lipschitz constants $L_{f,\mathcal{D}}$ in Proposition \ref{prop:epsilon_FOSP1} and \ref{prop:epsilon_FOSP2} with a simple global constant $L$. Nevertheless, throughout this paper, we do not require the global Lipschitz smoothness of $f$ thanks to our ability to implement line-search.
By Proposition \ref{prop:epsilon_FOSP1} and \ref{prop:epsilon_FOSP2} with $\epsilon=0$, we immediately have the following theorem on the equivalence of first-order minimax (stationary) points for \eqref{Prob_Ori} and \eqref{Prob_Min}.  
\begin{theo}
\label{thm:equivalence_FOSP}
Given Assumption \ref{Assumption_1} and $\beta > \mu^{-1}$, $(x^*, y^*)$ is a first-order minimax point of \eqref{Prob_Ori} if and only if it is a first-order stationary point of \eqref{Prob_Min}. 
\end{theo}

Next, we present the equivalence of global and local solutions for \eqref{Prob_Ori} and \eqref{Prob_Min}.

\begin{theo}
    \label{thm:equivalence-global-local}
   Given Assumption \ref{Assumption_1} and $\beta > \mu^{-1}$, $(x^*, y^*)$ is a global (local) minimax point of \eqref{Prob_Ori} if and only if it is a global (local) minimizer of \eqref{Prob_Min}. 
\end{theo}
\begin{proof}
    Suppose $(x^*, y^*)$ is a local minimizer of \eqref{Prob_Min}. Then there exists $\delta > 0$ s.t. 
    \[
    \hb(x^*, y^*) \leq \hb(x,y), \quad \forall (x,y)\in\cB_\delta(x^*,y^*).
    \]
    Furthermore, $(x^*,y^*)$ must be a first-order stationary point of \eqref{Prob_Min}. Then Theorem \ref{thm:equivalence_FOSP} implies that $\dfy(x^*,y^*) = 0$ and $y^* \in \ys(x^*)$. %, which further implies $y^* \in \ys(x^*)$ by concavity of $f(x^*,\cdot)$. 
    Therefore, one must have $\hb(x^*,y^*) = \Phi(x^*)$. Also note that $\ys$ is continuous by Lemma \ref{lem:properties-of-value-fcn}, and thus there exists $\delta' > 0$ such that $(x, \ys(x)) \in \cB_\delta(x^*,y^*)$ for all $x \in \cB_{\delta'}(x^*)$. Then it holds that
    \[
    \Phi(x^*) = \hb(x^*, y^*) \leq \hb(x,\ys(x)) = \Phi(x),\quad \forall x \in \cB_{\delta'}(x^*),
    \]
    which implies that $(x^*, y^*)$ is a local minimax point of \eqref{Prob_Ori}.

    Conversely, suppose $(x^*, y^*)$ is a local minimax point of \eqref{Prob_Ori}. Then, there exists $\delta > 0$ such that 
    \[
    \Phi(x^*) \leq \Phi(x),\;\; \forall x \in \cB_\delta(x^*)  \qaq y^* = \ys(x^*).
    \]
    Take any $(x, y) \in \cB_{\delta}(x^*, y^*)$. Then $x \in \cB_\delta(x^*)$ and by \eqref{eq:h-coer-inter1} it holds that
    \begin{align*}
        \hb(x,y) \geq \Phi(x) + \frac{\beta - \mu^{-1}}{2} \|\dfy (x,y)\|^2 \geq 
        \Phi(x) \geq \Phi(x^*) = \hb(x^*,y^*),
    \end{align*}
    which implies $(x^*, y^*)$ is a local minimizer of \eqref{Prob_Min}. The  global case follows similarly.
\end{proof}

In the end of this subsection, we establish the equivalence of \eqref{Prob_Ori} and \eqref{Prob_Min} with regard to their SONC and SOSC.

\begin{theo}
\label{thm:second-order-equvilance}
Given Assumption \ref{Assumption_1} and $\beta > \mu^{-1}$, $(x^*, y^*)$ satisfies the SONC (SOSC) of \eqref{Prob_Ori} if and only if it satisfies the SONC (SOSC) of \eqref{Prob_Min}. %Furthermore, $(x^*, y^*)$ satisfies the SOSC of \eqref{Prob_Ori} if and only if it satisfies the SOSC of \eqref{Prob_Min}.
\end{theo}
\begin{proof}
    For second-order equivalence, let $(x^*, y^*)$ be a first-order minimax point of \eqref{Prob_Ori} (or a first-order stationary point of \eqref{Prob_Min}).
    By \eqref{eq:formula-for-ddhyy}, it holds that
    \[
    \ddhyy(x^*, y^*) = \ddfyy (x^*, y^*) + \beta[\ddfyy (x^*, y^*)]^2.
    \]
    Let the eigen-decomposition of $\ddfyy (x^*, y^*)$ be $U D U^\top$, where $U$ is orthogonal and $D$ is a diagonal matrix with negative entries. Since $\beta > \mu^{-1}$ and each diagonal entry of $D$ is less than or equal to $-\mu$, we know $D + \beta D^2 \succ0$, which further implies that 
    \[ 
    \ddhyy(x^*, y^*) = 
    U[D + \beta D^2]U^\top \succ 0.
    \]
    Then, the Schur complement of $\nabla^2 \hb(x^*, y^*)$ in \eqref{eq:hessian-of-h} is
    \begin{align*}
        &\ddhxx(x^*, y^*) - \ddhxy (x^*, y^*)[\ddhyy(x^*, y^*)]^{-1} \ddhyx(x^*, y^*)\\
        =& \ddfxx(x^*, y^*) + \beta\ddfxy(x^*, y^*)\ddfyx(x^*, y^*) - \\&\hspace{1.5cm}
        \ddfxy(x^*, y^*)  [\ddfyy(x^*, y^*)]^{-1}[I_m + \beta\ddfyy(x^*, y^*)]\ddfyx(x^*, y^*)\\
        =& \ddfxx(x^*, y^*) - 
        \ddfxy(x^*, y^*)[\ddfyy(x^*, y^*)]^{-1}\ddfyx(x^*, y^*)
        = \nabla^2 \Phi(x^*).
    \end{align*}
    Hence, we have
    \[
    \nabla^2 \hb(x^*, y^*) \succ 0 \iff \nabla^2 \Phi(x^*) \succ 0
    \qaq
    \nabla^2 \hb(x^*, y^*) \succeq 0 \iff \nabla^2 \Phi(x^*) \succeq 0.
    \]
    This establishes the second-order equivalence and completes the proof.
\end{proof} 

\subsection{Line-search Framework and Global Convergence Properties}
\label{subsec:line-search-framework}
With the previous preparation, we are ready to formally introduce and analyze our line-search framework. In each iteration $k$, the updates take the form
\begin{align*}
    \ykp ={}& \yk + \eyk \nabla_y f(\xk, \yk),\\
    \xkp ={}& \xk + \exk \dxk,
\end{align*}
where  
$\eyk, \exk > 0$ are the step sizes determined by line-search and $\dxk \in \Rn$ is the search direction for the primal variable $x$. Note that for gradient-descent-ascent type algorithms, the two step sizes $\eta_{y,k}$ and $\eta_{x,k}$ often need to have different scales, see e.g. \cite{lin2020gradient}, and both of them should be searched. Yet in the backtracking line-search context, it is not a good idea to do grid search for two parameters simultaneously. Therefore, in this paper, the above update scheme is carried out in an alternating style. That is, at $(x_k,y_k)$, we make a backtracking line-search to determine the parameter $\eta_{y,k}$ and obtain $y_{k+1}$ first, and then we construct the direction $d_{x,k}$ at the point $(x_k,y_{k+1})$ and make a backtracking line-search to determine $\eta_{x,k}$. In particular, we assume that $\dxk$ satisfies the following assumption, which shall be verified in later discussion with detailed construction of $\dxk$. 

\begin{assumpt}%[\bf Search directions]
    \label{ass:search-directions}  
    For every $k \geq 0$, there exists $a_1,a_2 > 0$ such that 
    \begin{equation}\label{eq:search-dir-ass-x}
    \begin{split}
         \|\dxk\| &\leq a_1 \|\nabla_x f(\xk, \ykp)\|\\ \inner{\dxk ,\dfx(\xk,\ykp)} &\leq -a_2 \|\dfx (\xk, \ykp)\|^2.
    \end{split}
    \end{equation}
\end{assumpt} 
\begin{rmk}
We make two remarks regarding this assumption.
\begin{enumerate}
    \item The requirement \eqref{eq:search-dir-ass-x} for the search direction in the $x$-variable is exactly the gradient-related search directions in line-search-based methods for minimization problems, see, e.g., \cite{grippo1986nonmonotone,zhang2004nonmonotone,cartis2015worst,grapiglia2021generalized}. It can cover a variety of practical search directions, see \cite[Section 1.2]{Bertsekas2016NonlinearProgramming}.
    \item Since $f$ is strongly concave in the $y$-variable, which admits considerably more structure than in the $x$-variable, it suffices to focus on the gradient-ascent search direction in $y$. Nevertheless, our line-search framework can accommodate any search directions $\dyk \in \Rm$ satisfying  
    \[
    \inner{\dhy(\xk,\yk), \dyk} \leq -c \|\dfy(\xk, \yk)\|^2,  \quad \text{for some } c > 0.
    \]
\end{enumerate} 
\end{rmk}
 
In the following, we introduce a bounded iterates assumption. This assumption and its variants (such as bounded level sets or directly assuming $\nabla f$ to be globally Lipschitz) are common in the literature for nonmonotone line-search methods, see, e.g., \cite{grippo1986nonmonotone,grippo1989truncated,zhang2004nonmonotone,cartis2015worst,lu2017randomized,grapiglia2021generalized}.  

\begin{assumpt}%[\bf Bounded iterates]
    \label{ass:bounded-iterates}
    The iterates $\{(\xk, \yk)\}_k$ generated by our line-search framework are bounded. That is, there exists $M > 0$ s.t. $(\xk, \yk) \in \cB_{M}(0)$, %\subseteq \Rn \times \Rm$, 
    for all $k \geq 0$.
\end{assumpt}

\begin{rmk}
As a result of Assumption \ref{Assumption_1} and Lemma \ref{lem:properties-of-h}, both $f$ and $\hb$ are twice continuously differentiable. One direct consequence of Assumption \ref{ass:bounded-iterates} is that the gradients $\df$ and $\dhb$ are Lipschitz continuous over the ball  $\cB_{M}(0)$, with Lipschitz moduli denoted by $\sL$ and $\sLh$, respectively.
\end{rmk} 

Before introducing our line-search schemes, we analyze the inner products between partial gradients of the Lyapunov function $\hb$ and our search directions $\dyk$ and $\dxk$.

\begin{lem}
    \label{lem:inner-product-bds}
    Suppose Assumption \ref{Assumption_1} and \ref{ass:bounded-iterates} hold and $\beta > \mu^{-1}$. It holds that
    \begin{align}
        \inner{ \nabla_y \hb(\xk, \yk), \dfy(\xk, \yk)} &\leq -b_1 \|\dfy(\xk, \yk)\|^2, 
        \label{eq:search-dir-requi-y}\\
        \inner{ \nabla_x \hb(\xk, \ykp), \dxk} &\leq -b_2 \|\nabla_x f(\xk, \ykp)\|^2 + b_3 \|\nabla_y f(\xk, \ykp)\|^2,\label{eq:search-dir-requi-x}
    \end{align}
    where $b_1 := \beta \mu - 1$, $b_2 := \frac{a_2}{2}$, $b_3 := \frac{(\sL \beta a_1)^2}{2a_2}.$
\end{lem}
\begin{proof}
    The inequality \eqref{eq:search-dir-requi-y} is exactly the first inequality in Lemma \ref{lem:various-gradient-bounds2}. %\eqref{eq:search-dir-bd-gda-y}. 
    To show \eqref{eq:search-dir-requi-x}, observe that
    \begin{align*}
        &2\inner{ \nabla_x \hb(\zkh), \dxk}
        = 2\inner{ \dfx(\zkh) + \beta \ddfyx(\zkh) \dfy(\zkh), \dxk}\\
        \leq& -2a_2 \|\dfx(\zkh)\|^2 + 2\inner{\beta\ddfyx(\zkh) \dfy(\zkh), \dxk}\\
        \leq& -2a_2 \|\dfx(\zkh)\|^2 + {\beta^2 a_1^2}{a_2^{-1}} \|\ddfyx(\zkh) \dfy(\zkh)\|^2 + {a_2}{a_1^{-2}} \|\dxk\|^2\\
        \leq& -{a_2} \|\dfx(\zkh)\|^2 + {\beta^2 \sL^2 a_1^2}{a_2^{-1}} \|\dfy(\zkh)\|^2,
    \end{align*}
    where the first inequality follows from \eqref{eq:search-dir-ass-x}, the second inequality is due to Young's inequality, and the last line is by \eqref{eq:search-dir-ass-x} and the $\sL$-smoothness of $f$ within the region $\mathcal{B}_M(0)$, see Assumption \ref{ass:bounded-iterates}.
\end{proof}

For the line-search criteria, in this paper, we mainly consider the ZH-type (Zhang-Hager) nonmonotone modified Armijo line-search rule \cite{zhang2004nonmonotone,hager2005new}, as specified in \eqref{eq:linesearch-condition-y}--\eqref{eq:ZH-weighted-fcn-value} of Algorithm \ref{alg:line-search-framework}. It is clear that $\tk \equiv 1$ corresponds to the classic Armijo line-search \cite{armijo1966minimization}. We use the gradient norm of $f$ to measure descent instead of the inner products $\inner{\dhy(\xk,\yk), \dfy(\xk,\yk)}$ and $\inner{\dhx(\xk,\ykp), \dxk}$ for two reasons. First, for reasonable search direction such as $\dxk = -\dfx(\xk,\ykp)$, it does not necessarily satisfy $\inner{\dhx(\xk,\ykp), \dxk} < 0$.
Second, the computation of $\nabla f$ is cheaper than that of $\dhb$, which already involves the computation of $\nabla f$ and a Hessian-vector product of $f$.

\begin{algorithm}[h]
\caption{Line-search framework for minimax problem \eqref{Prob_Ori}.}
\label{alg:line-search-framework}
\begin{algorithmic}[1]
\State{\textbf{Initialization:} $x_0 \in \Rn$, $y_0 \in \Rm$, $\beta > \mu^{-1}$, $k = 0$, $\eta_x, \eta_y > 0$, $\alpha \in (0, 1)$, $0 \leq \gamma_x < \gamma_y < 1$, $\tau \in (0, 1]$, $\{\tau_k\}_k \subseteq [\tau, 1]$, and $H_0 = \hb(x_0, y_0)$.}
%parameters in conditions \eqref{eq:linesearch-condition-y}, \eqref{eq:linesearch-condition-x}, and \eqref{eq:ZH-weighted-fcn-value}.}
\While{termination criteria not met}
\State{Choose $\eyk$ to be the largest element in $\{\eta_y \alpha^l: l = 0,1,2,\ldots\}$ satisfying
\begin{equation} \label{eq:linesearch-condition-y}
    \hb(\xk, \yk + \eyk \dfy(\xk,\yk)) \leq \Hk -\gamma_y b_1 \eyk \|\nabla_y f(\xk, \yk)\|^2. 
\end{equation}}
\vspace{-4mm}
\State{$\ykp \gets \yk + \eyk \dfy(\xk,\yk)$.}%, where $\eyk$ is determined by the line-search scheme \eqref{eq:linesearch-condition-y}}.
\State{Compute a search direction $\dxk$ satisfying \eqref{eq:search-dir-requi-x}.}
\State{Choose $\exk$ to be the largest element in $\{\eta_x \alpha^l: l = 0,1,2,\ldots\}$ satisfying
\begin{equation}  \label{eq:linesearch-condition-x}
\begin{split}
    %\xkp = \xk + \exk \dxk(\xk, \ykt), \quad \text{where} \quad \ykt = \yk \text{ or } \ykp, \\
    \hb(\xk + \exk \dxk, \ykp) &\leq 
    \Hk -  \\ &\hspace{-2cm}\gamma_x \big(b_1\eyk \|\nabla_y f(\xk, \yk)\|^2 + b_2\exk \|\nabla_x f(\xk, \ykp)\|^2  \big), 
\end{split}
\end{equation}}
\vspace{-2mm}
\State{$\xkp \gets \xk + \exk \dxk$.} %, where $\exk$ is determined by the line-search procedure \eqref{eq:linesearch-condition-x}.}
\State{Update the weighted average $H_{k+1}$ with \begin{equation}
    \label{eq:ZH-weighted-fcn-value}
    \Hkp = (1-\tk)\Hk + \tk \hb(\xkp, \ykp).
\end{equation}}
%\eqref{eq:ZH-weighted-fcn-value}}.
\vspace{-4mm}
\State{$k \gets k+1$}.
\EndWhile
\end{algorithmic}
\end{algorithm}

% That is, the step size $\eyk$ is the largest element in $\{\eta_y \alpha^l: l = 0,1,2,\ldots\}$ satisfying
% \begin{align} \label{eq:linesearch-condition-y}
%     \hb(\xk, \yk + \eyk \dfy(\xk,\yk)) &\leq \Hk -\gamma_y b_1 \eyk \|\nabla_y f(\xk, \yk)\|^2, %\quad \text{where} \quad \dyk = \dfy(\xk,\yk),%\inner{ \nabla_y \hb(\xk, \yk), \dyk^k},
% \end{align}
% and the step size in the $x$ direction $\exk$ is the largest element in $\{\eta_x \alpha^l: l = 0,1,2,\ldots\}$ satisfying
% \begin{equation}  \label{eq:linesearch-condition-x}
% \begin{gathered}
%     %\xkp = \xk + \exk \dxk(\xk, \ykt), \quad \text{where} \quad \ykt = \yk \text{ or } \ykp, \\
%     \hb(\xk + \exk \dxk, \ykp) \leq 
%     \Hk -  \gamma_x \Big(b_1\eyk \|\nabla_y f(\xk, \yk)\|^2 + b_2\exk \|\nabla_x f(\xk, \ykp)\|^2  \Big), 
% \end{gathered}
% \end{equation}
% where $\eta_x, \eta_y > 0$, $\alpha \in (0, 1)$, $0 \leq \gamma_x < \gamma_y < 1$, $H_0 = \hb(x_0, y_0)$, and for $k \geq 0$,
% \begin{equation}
%     \label{eq:ZH-weighted-fcn-value}
%     \Hkp = (1-\tk)\Hk + \tk \hb(\xkp, \ykp), \quad \tk \in [\tau, 1] \text{ for some } \tau \in (0, 1].
% \end{equation}

The nonmonotone line-search schemes \eqref{eq:linesearch-condition-y}--\eqref{eq:ZH-weighted-fcn-value} can also serve as a globalization strategy for adaptive step sizes such as the Barzilai-Borwein (BB) step size \cite{barzilai1988two,raydan1997barzilai}, which can often lead to fast convergence in practical applications, see Section \ref{subsec:BB} and Section \ref{sec:numerical experiments}. Other line-search conditions, such as Armijo-Wolfe line-search \cite{wolfe1969convergence,wolfe1971convergence}, Armijo-Goldstein line-search \cite{goldstein1965steepest}, and Grippo-Lampariello-Lucidi-type nonmonotone line-search \cite{grippo1986nonmonotone}, can be incorporated into our framework similarly.

Synthesizing the preceding discussions, we formally introduce our line-search framework Algorithm \ref{alg:line-search-framework} for minimax problem \eqref{Prob_Ori}. A simple termination criteria for the line-search framework would be $\|\df(\xk, \yk)\| \leq \varepsilon$. In the remainder of this section, we always assume $\beta > \mu^{-1}$ for the convenience of discussion. In Section \ref{subsec:parameter-free}, we present a procedure for adaptively estimating $\mu$ (and $\beta$) and a fully parameter-free method with robust theoretical guarantees.

We now provide a few basic properties of our line-search scheme \eqref{eq:linesearch-condition-y}--\eqref{eq:ZH-weighted-fcn-value}. 

\begin{prop}
    \label{prop:linesearch-terminates-finitely}
    Suppose Assumption \ref{Assumption_1}, \ref{ass:search-directions}, and \ref{ass:bounded-iterates} hold and $\beta > \mu^{-1}$. The following statements hold.
    \begin{enumerate}
        \item The sequence $\{\Hk\}_k$ is nonincreasing with $H_k \geq \hb(\xk,\yk)$, for each $ k \geq 0$.
        \item The backtracking procedures in \eqref{eq:linesearch-condition-y} and \eqref{eq:linesearch-condition-x} both terminate finitely for every $k \geq 0$.
        \item There exists $\bar \eta_x, \bar \eta_y > 0$ such that $\exk \geq \bar \eta_x$ and $\eyk \geq \bar \eta_y$, for all $k \geq 0$.
        \item There exists a real number $\bar H \geq \bar \Phi$ such that
        \[
        \lim_{k \to \infty} H_k = \lim_{k \to \infty} \hb(\xk, \yk) = \bar H.
        \]
    \end{enumerate}
\end{prop}
\begin{proof}
    We first prove part 1, 2, and 3 jointly by induction. By definition, we have $H_0 = \hb(z_0)$. Suppose we have established $H_k \geq \hb(\zk)$ for some $k \geq 0$. We now prove that the line-search procedures in \eqref{eq:linesearch-condition-y} and \eqref{eq:linesearch-condition-x} terminate finitely in iteration $k$. Since $\hb$ is $\sLh$-smooth over $\cB_{M}(0)$, by the standard descent lemma, it holds that 
    \vspace{-1.5ex}
    \begin{align*}
        \hb(\xk, \yk + \eta \dfy(\zk)) 
        &\leq \hb(z_k) + \eta \inner{\dhy(\zk), \dfy(\zk)} + \frac{\sLh\eta^2}{2} \|\dfy(\zk)\|^2\\
        &\leq \hb(z_k) - \eta\Big(b_1 - \frac{\sLh}{2}\eta\Big) \|\dfy(\zk)\|^2.
    \end{align*}
    where the last line follows from \eqref{eq:search-dir-requi-y}. 
    Since $\hb(z_k) \leq \Hk$, the line-search condition \eqref{eq:linesearch-condition-y} can be fulfilled by step size $\eta \in (0, \eta_y]$ satisfying
    \begin{equation*}
    %\label{eq:step-size-insight1}
        b_1 - \frac{\sLh}{2}\eta \geq \gamma_y b_1 \iff \eta \leq \frac{2 b_1 (1-\gamma_y)}{\sLh}
    \end{equation*}
    and thus the backtracking for \eqref{eq:linesearch-condition-y} terminates finitely in iteration $k$.
    Furthermore, the backtracking procedure would guarantee
    \begin{equation}
        \label{eq:step-size-y-lower-bd}
        \eyk \geq \min\left\{\frac{2 \alpha b_1 (1-\gamma_y)}{\sLh}, \eta_y\right\} =: \eyb > 0.
    \end{equation}
    %where $\alpha \in (0,1)$ is the backtracking factor and $\eta_y > 0$ is the initial step size for line-search.
    
    %If $\ykt = \ykp$ in \eqref{eq:linesearch-condition-x}, 
    \noindent Now, for the line-search procedure in \eqref{eq:linesearch-condition-x}, again, by the descent lemma, it holds that
    \begin{equation*}
        %\label{eq:prop-inter1}
        \begin{aligned}
        &\hb(\xk + \eta\dxk,\ykp) 
        \leq \hb(\zkh) + \eta \inner{\dhx(\zkh), \dxk} + \frac{\sLh\eta^2 }{2} \|\dxk\|^2 \\
        \leq& \hb(\zkh) - \eta\left[b_2 \|\dfx(\zkh)\|^2 - b_3 \|\dfy(\zkh)\|^2\right] + \frac{\sLh\eta^2a_1^2}{2} \|\dfx(\zkh)\|^2 \\
        \leq& \Hk - \gamma_y b_1 \eyk \|\dfy(\zk)\|^2 + \eta b_3 \|\dfy(\zkh)\|^2 - \eta \Big(b_2 - \frac{\sLh a_1^2}{2} \eta\Big)\|\dfx(\zkh)\|^2, 
    \end{aligned}
    \end{equation*}
    where the second inequality follows from \eqref{eq:search-dir-requi-x} and \eqref{eq:search-dir-ass-x} and the last inequality is due to \eqref{eq:linesearch-condition-y}. 
    Now, we investigate the relationship between $\|\dfy(\zkh)\|^2$ and $\|\dfy(\zk)\|^2$ utilizing triangular inequality and the $L_f$-Lipschitz continuity of $\df$ over $\cB_M(0)$:
    \begin{equation*}
    %\label{eq:prop-inter2}
        \begin{aligned}
        \|\dfy(\zkh)\| &=
        \|\dfy(\zkh) - \dfy(\zk) + \dfy(\zk)\| \\
        &\leq \sL \|\ykp - \yk\| +  \|\dfy(\zk)\| 
        = (\sL \eyk + 1) \|\dfy(\zk)\|. 
        \end{aligned}
    \end{equation*}
    Combining this inequality with our previous descent lemma-based estimation yields
    \begin{align*}
        \hb(\xk + \eta\dxk,\ykp) &\leq 
        \Hk - \eta \Big(b_2 - \frac{\sLh a_1^2}{2} \eta\Big)\|\dfx(\zkh)\|^2  \\ &\hspace{2cm} 
         -\left[\gamma_y b_1 \eyk-  b_3 (\sL \eyk + 1)^2 \eta\right] \|\dfy(\zk)\|^2.
    \end{align*}
    Then the line-search condition \eqref{eq:linesearch-condition-x} can be fulfilled for $\eta \in (0, \eta_x]$ satisfying
    \begin{equation}
    \label{eq:step-sizes-insight2}
        b_2 - \frac{\sLh a_1^2}{2} \eta \geq \gamma_x b_2 \qaq \gamma_y b_1 \eyk -b_3 (\sL \eyk + 1)^2 \eta \geq \gamma_x b_1 \eyk.
    \end{equation}
    Since $\gamma_x < \gamma_y$, $\eyk \in [\eyb, \eta_y]$ and $\eta \in (0, \eta_x]$, a sufficient condition for \eqref{eq:step-sizes-insight2} is
    \begin{equation*}
        %\label{eq:prop-inter4}
        \eta \leq \min\left\{ \frac{2(1-\gamma_x)b_2}{a_1^2 \sLh},  \frac{(\gamma_y-\gamma_x) b_1 \eyb}{b_3 (\sL\eta_y + 1)^2}\right\} := \eta'_x.
    \end{equation*}
    Hence, the backtracking for \eqref{eq:linesearch-condition-x} terminates finitely in iteration $k$ and the backtracking procedure would guarantee
    \begin{equation*}
        \exk \geq 
        \min\{\alpha \eta_x',\eta_x\}
        := \exb > 0.
    \end{equation*} 
    Now, observe that Lemma \ref{lem:properties-of-h}, \eqref{eq:linesearch-condition-x}, and \eqref{eq:ZH-weighted-fcn-value} imply 
    \[
    \bar \Phi \leq \hb(\zkp) \leq \Hk
    \qaq
    \hb(\zkp) \leq
    \Hkp = (1-\tk) \Hk + \tk \hb(\zkp) \leq H_k.
    \]
    Hence, by induction, part 1, 2, and 3 hold. For part 4, since $\{\Hk\}_k$ is bounded below by $\bar \Phi$ and nonincreasing, there exists $\bar H \geq \bar \Phi$ such that $H_k \to \bar H$. Moreover, note that
    \[
    0 \leq \Hk - \hb(\zkp) = \tk^{-1} (\Hk - \Hkp) \leq \tau^{-1} (\Hk - \Hkp) \to 0,
    \]
    which completes the proof.
\end{proof}

% \begin{rmk}
%     Since the line-search conditions \eqref{eq:linesearch-condition-y} and \eqref{eq:linesearch-condition-x} only require knowledge on the constants $c_2$ and $c_4$ in Assumption \ref{ass:search-directions}, which usually do not depend on the global Lipschitz constant $\sL$ of $\nabla f$ (see Example \ref{exp:update-dir-bds}), all of our preceding discussions remain to be valid if we only assume $f$ to be locally Lipschitz continuous. In this case, although the conditions \eqref{eq:search-dir-requi-y} and \eqref{eq:search-dir-requi-x} might only hold in a local neighborhood of $(\xk, \yk)$ and the constants $c_1, c_3, c_5$ might depend on the local Lipschitz constant of $\nabla f$ over this neighborhood, the line-search procedure can automatically adapt to the local Lipschitz constant and terminal finitely. (The proof is also the same, with all the constants replaced by the corresponding local versions.) We choose to work with globally $\sL$-smooth $f$ for the ease of presentation. 
% \end{rmk}

\begin{rmk}
    \label{rmk:bdd-iter}
    There are two important observations regarding Proposition \ref{prop:linesearch-terminates-finitely}.
    \begin{enumerate}
        \item By Proposition \ref{prop:linesearch-terminates-finitely}, the line-search scheme \eqref{eq:linesearch-condition-y}--\eqref{eq:ZH-weighted-fcn-value} guarantees \[(\xk, \yk) \in [\hb(x,y) \leq \hb(x_0, y_0)] =: \Omega_{0}, \quad \forall k \geq 0.\] Hence, Assumption \ref{ass:bounded-iterates} can be ensured if $\Omega_0$ is bounded. By Lemma \ref{lem:coercivity-of-h}, when $\beta > \mu^{-1}$, this can be implied by coercivity of the value function $\Phi$, or equivalently, the boundedness of the level set $[\Phi(x) \leq \hb(x_0, y_0)]$.
        \item Suppose $\dxk = -\dfx(\xk, \ykp)$ and $\beta = 2\mu^{-1}$. Then, by direct computation and the first inequality in Lemma \ref{lem:various-gradient-bounds1}, %\eqref{eq:search-dir-bd-gda-x}, 
        the constants in \eqref{eq:search-dir-ass-x} and \eqref{eq:search-dir-requi-x} are
        %begin{equation*}
            %\label{eq:param-gda-dx}
            $a_1 = a_2 = b_1 = 1, %\quad 
            b_2 = 1/2, %\qaq 
            \text{ and } b_3 = 2\kappa^2,$
        %\end{equation*}
        where $\kappa = \sL/\mu$. If $\eyk \sim 1/\sL$, the second inequality in \eqref{eq:step-sizes-insight2} and the backtracking imply that %for some constant $c > 0$
        \[
        \exk / \eyk \geq O(1/\kappa^2).
        \]
        This roughly matches the prescribed step sizes in two-timescale GDA method \cite[Theorem 17]{lin2025two}, but it is automatically determined by our line-search mechanism without explicit knowledge of $\sL$ (since our line-search scheme only requires $b_1 = 1, b_2 = 1/2$). Nevertheless, the nonmonotone line search mechanism can often allow more aggressive step sizes, providing a much faster practical convergence, compared to the standard two-timescale GDA with prescheduled step sizes. 
    \end{enumerate}
\end{rmk}

\noindent In the remainder of this subsection, we establish the global convergence of Algorithm \ref{alg:line-search-framework}. And we show that $\nabla f(\xk, \yk) \to 0$ at an $O(T^{-1/2})$ sublinear rate. 

\begin{theo}
    \label{thm:glob-conv-iter-compl}
    Suppose Assumption \ref{Assumption_1}, \ref{ass:search-directions}, and \ref{ass:bounded-iterates} hold and $\beta > \mu^{-1}$. Then the iterates $\{(\xk, \yk)\}_k$ generated by Algorithm \ref{alg:line-search-framework} satisfy 
    \[
    \nabla f(\xk, \yk) \to 0 \qaq \nabla \Phi(\xk) \to 0,
    \]
    as $k \to \infty$. %In particular, every accumulation point of $\{(\xk, \yk)\}_k$ must be a first-order minimax point of \eqref{Prob_Ori}. 
    Furthermore, it holds that
    \[
    \min_{k=0,1,\ldots,T} \|\nabla f(\xk, \yk)\| = O(T^{-1/2}) \qaq \min_{k=0,1,\ldots,T} \|\nabla \Phi(\xk)\| = O(T^{-1/2}).
    \]
\end{theo}
\begin{proof}
    The line-search condition \eqref{eq:linesearch-condition-x} and update rule of $\Hk$ \eqref{eq:ZH-weighted-fcn-value} imply that
    \begin{equation*}
        \begin{aligned}
            \Hkp &= \Hk - \tk \left(\Hk - \hb(\zkp)\right) \\
            &\leq \Hk -  \tk\gamma_x\left(b_1\eyk \|\dfy(\zk)\|^2 + b_2\exk \|\dfx(\zkh)\|^2\right).
        \end{aligned}
    \end{equation*}
    Combining this inequality, Proposition \ref{prop:linesearch-terminates-finitely} part 3, and the fact that $\tk \geq \tau > 0$ gives
    \begin{align*}
        \Hkp %&\leq \Hk -  \frac{\tk\gamma_x}{2}[c_4\exk \|\dfx(\zkt)\|^2 + c_2\eyk \|\dfy(\zk)\|^2] \\
        &\leq \Hk - \tau\gamma_x\left(b_1\eyb \|\dfy(\zk)\|^2 + b_2\exb \|\dfx(\zkh)\|^2\right).
    \end{align*}
    Telescoping this inequality over $k = 0, 1, \ldots, T$ and rearranging the terms yield 
    \begin{equation*}
        %\label{eq:iter-comp-inter1}
        b_1\eyb {\sum}_{k=0}^T \|\dfy(\zk)\|^2 + b_2\exb {\sum}_{k=0}^T \|\dfx(\zkh)\|^2 
        \leq \frac{2(H_0 - H_{T+1})}{\tau \gamma_x} \leq \frac{2(H_0 - \bar H)}{\tau \gamma_x},
    \end{equation*}
    where the last inequality is due to part 1 and 4 of Proposition \ref{prop:linesearch-terminates-finitely}.
    Now, note that
    \begin{equation*}
    %\label{eq:iter-comp-inter2}
    \begin{aligned}
        \|\dfx(\zk)\|^2 
        &= \|\dfx(\zk) - \dfx(\zkh) + \dfx(\zkh)\|^2 \\ 
        &\leq 2\sL^2\|\yk - \ykp\|^2 + 2\|\dfx(\zkh)\|^2 \\
        &\leq 2(\sL\eta_y)^2 \|\dfy(\zk)\|^2 + 2\|\dfx(\zkh)\|^2, 
    \end{aligned}
    \end{equation*}
    where the first inequality follows from Young's inequality and Lipschitz continuity of $\df$ and the second inequality is due to $\eyk \leq \eta_y$. 
    Combining the previous two estimations yields
    \begin{align*}
        &\min\{b_1 \eyb, b_2 \exb\} {\sum}_{k=0}^T \|\nabla f(\zk)\|^2\\
        \leq& b_1\eyb{\sum}_{k=0}^T \|\dfy(\zk)\|^2 + b_2\exb {\sum}_{k=0}^T \|\dfx(\zk)\|^2 \\
        \leq&  
        \left[b_1\eyb + 2b_2\exb(\sL \eta_y)^2\right]{\sum}_{k=0}^T \|\dfy(\zk)\|^2 + 2 b_2\exb {\sum}_{k=0}^T \|\dfx(\zkh)\|^2\\
        \leq& \left[1 + \frac{b_2\exb(\sL \eta_y)^2}{b_1 \eyb}\right] \frac{4(H_0 - \bar H)}{\tau \gamma_x}, \quad \forall T \geq 0,
    \end{align*}
    which implies $\nabla f(\zk) \to 0$ and its global convergence rate. 
    Moreover, the global convergence properties in terms of $\|\nabla \Phi(\xk)\|$ can be inferred from those of $\|\nabla f(\zk)\|$ by observing that
    \begin{align*}
        \|\nabla \Phi(\xk)\| &= \|\dfx(\xk,\ys(\xk))\| = \|\dfx(\xk,\ys(\xk)) - \dfx(\zk) + \dfx(\zk)\|\\
        &\leq \sL\|\ys(\xk) - \yk\| + \|\dfx(\zk)\| \leq 
        {\sL}{\mu^{-1}} \|\dfy(\zk)\| + \|\dfx(\zk)\|,
    \end{align*}
    where the first inequality follows from the triangular inequality and Lipschitz continuity of $\df$ and the second inequality is due to $\mu$-strongly concavity of $f(\xk,\cdot)$
\end{proof}

\subsection{Iterate Convergence and Local Rates under the {\L}ojasiewicz Property} \label{sec:KL-iterate-convergence}
In this subsection, we prove the asymptotic convergence of the full iterate sequence $\{(\xk, \yk)\}_k$ as well as the sequence's local convergence rates for Algorithm \ref{alg:line-search-framework} under the KL property (see Definition \ref{def:Loj-inequality}). To start with, let us first establish the equivalence between the KL property of the value function $\Phi$ and the KL property of the Lyapunov function $\hb$.

\begin{theo}
\label{thm:KL-equivalence}
Suppose Assumption \ref{Assumption_1} holds. Take any $\beta > 0$ and $\theta \in \big[\frac12,1\big)$.
\begin{enumerate}
    \item If the function $\hb$ satisfies the KL property at $(\bar x, \ys(\bar x))$ with exponent $\theta$, then the value function $\Phi$ satisfies the KL property at $\bar x$ with the same exponent $\theta$.
    \item Conversely, suppose the value function $\Phi$ satisfies the KL property at $\bar x$ with exponent $\theta$. Then $\hb$ satisfies the quasi-KL property: there exists $C_{\hb} > 0$ and a neighborhood  $U(\bar x,\ys (\bar x))$ around $(\bar x,\ys (\bar x))$ such that
    \begin{equation} \label{eq:quasi-Loj}
        |\hb(x,y) - \hb(\bar x, \ys(\bar x))|^\theta \leq C_{\hb} \|\nabla f(x,y)\|, \quad \forall (x,y) \in U(\bar x,\ys (\bar x)).
    \end{equation}
    If additionally $\beta > \mu^{-1}$, then $\hb$ also satisfies the KL inequality at $(\bar x,\ys (\bar x))$ with the same exponent $\theta$.
\end{enumerate}
\end{theo}
\begin{proof}
Suppose $\hb$ satisfies the KL property at $(\bar x, \ys(\bar x))$. There exists $C_{1}, \varepsilon_1 > 0$ such that
\begin{equation}
\label{eq:h-beta-Loj}
    C_{1} \|\nabla \hb(x,y)\| \geq  |\hb(x,y) - \hb(\bar x, \ys(\bar x))|^\theta, \quad \forall (x,y) \in \cB_{\varepsilon_1}(\bar x, \ys(\bar x)).
\end{equation}
Note that for each $x \in \Rn$,
\begin{equation}
\label{eq:Phi-h-beta-relation}
    \|\nabla \hb(x, \ys(x))\| = \|\nabla_x f(x, \ys(x))\|= \|\nabla \Phi(x)\| \quad \text{and} \quad \hb(x,\ys(x)) = \Phi(x).
\end{equation}
By Lemma \ref{lem:properties-of-value-fcn}, $\ys$ is locally Lipschitz continuous. Hence, there exists $\varepsilon_1' > 0$ s.t. 
\begin{align*}
    x \in \cB_{\varepsilon_1'}(\bar x) &\implies (x, \ys(x)) \in \cB_{\varepsilon_1}(\bar x, \ys(\bar x))\\
    & \overset{\eqref{eq:h-beta-Loj}}{\implies}  C_{1}\|\nabla \hb(x,\ys(x))\| \geq  |\hb(x,\ys(x)) - \hb(\bar x, \ys(\bar x))|^\theta\\
    & \overset{\eqref{eq:Phi-h-beta-relation}}{\implies} C_{1} \|\nabla \Phi(x)\| \geq  |\Phi(x) - \Phi(\bar x)|^\theta,
\end{align*}
%where the second line follows from \eqref{eq:h-beta-Loj} and the third line from \eqref{eq:Phi-h-beta-relation}. 
which proves part 1.

Conversely, suppose $\Phi$ satisfies the KL property at $\bar x$. There exists $C_2, \varepsilon_2 > 0$ such that 
\begin{equation} \label{eq:value-function-Loj}
    C_2  \|\nabla \Phi(x)\| \geq |\Phi(x) - \Phi(\bar x))|^\theta, \quad \forall x \in \cB_{\varepsilon_2}(\bar x).
\end{equation}
Choose $\varepsilon_2' \in (0, \varepsilon_2)$ sufficiently small s.t. $\|\nabla_y f(x,y)\| \leq 1$ for all $(x,y) \in \cB_{\varepsilon_2'}(\bar x, \ys (\bar x))$ and let $\bar{\mathsf{L}}_f$ be the Lipschitz modulus of $\nabla_x f$ over $\cB_{\varepsilon_2'}(\bar x, \ys (\bar x))$.
By the subadditivity $(u+v)^\theta \leq u^\theta + v^\theta$ for $u,v \geq 0$ and $\theta \in [\frac{1}{2},1)$, it holds that for any $(x,y) \in \cB_{\varepsilon_2'}(\bar x, \ys (\bar x))$,
\begin{align*} %\label{eq:intermediate-h-beta-Phi}
&|\hb(x,y) - \hb(\bar x, \ys(\bar x))|^\theta\\
\leq&
|\hb(x,y) - \hb(x, \ys(x))|^\theta + |\hb(x,\ys(x)) - \hb(\bar x, \ys(\bar x))|^\theta \nonumber\\
\leq&  |f(x,y) - f(x, \ys(x))|^\theta + \beta^{\theta} 2^{-\theta} \|\nabla_y f(x, y)\|^{2\theta} + C_2 \|\nabla\Phi(x)\|\nonumber\\
\leq& \underbrace{[(2\mu)^{-\theta} + \beta^{\theta} 2^{-\theta}]}_{=:C_3} \|\nabla_y f(x,y)\|^{2\theta} + C_2 \|\nabla_x f(x, \ys(x))\|\\
\leq& C_3 \|\nabla_y f(x,y)\| + C_2 \|\nabla_x f(x, \ys(x)) - \nabla_x f(x, y) \| + C_2 \|\nabla_x f(x, y)\|\\
\leq& C_3 \|\nabla_y f(x,y)\| + C_2 \bar{\mathsf{L}}_f \|\ys(x) - y\| + C_2 \|\nabla_x f(x, y)\|\\
\leq& C_3 \|\nabla_y f(x,y)\| + C_2 \bar{\mathsf{L}}_f \mu^{-1} \|\nabla_y f(x, y)\| + C_2 \|\nabla_x f(x, y)\|,
\end{align*}
where the second inequality follows from the subadditivity and \eqref{eq:value-function-Loj}, the third inequality is due to strong concavity of $f(x,\cdot)$ and \eqref{eq:Phi-h-beta-relation}, the fourth inequality follows from $\|\nabla_y f(x,y)\| \leq 1$ and $\theta \geq 1/2$, the fifth inequality is due to the local Lipschitz continuity of $\nabla f$, and the last inequality follows from the strong concavity of $f(x, \cdot)$. This establishes the quasi-KL property \eqref{eq:quasi-Loj} with $U(\bar x, \ys(\bar x)) =  \cB_{\varepsilon_2'}(\bar x, \ys (\bar x))$. 

Now, suppose $\beta > \mu^{-1}$. To complete the proof of part 2, by the second inequality in Lemma \ref{lem:various-gradient-bounds2} %\eqref{eq:lower-bd-gradient-hy} 
and the second inequality in Lemma \ref{lem:various-gradient-bounds1} with $C := 1 + \frac{(\beta\mu -1)^2}{2\beta^2  \bar{\mathsf{L}}_f^2} > 1$, it holds that for every  $(x, y) \in \cB_{\varepsilon_2'}(\bar x, \ys(\bar x))$
\begin{align}
    \|\nabla \hb(x,y)\|^2 &= \|\nabla_x \hb(x,y)\|^2 + \|\nabla_y \hb(x,y)\|^2 \nonumber\\
    &\geq (1-C^{-1}) \|\nabla_x f(x,y)\|^2 + \frac{(\beta\mu-1)^2}{2} \|\nabla_y f(x,y)\|^2. \nonumber \\
    &\geq \min\bigg\{1-C^{-1}, \frac{(\beta\mu-1)^2}{2}\bigg\} \cdot \|\df(x,y)\|^2. \nonumber
\end{align}
Combining this inequality and the quasi-KL property \eqref{eq:quasi-Loj} completes the proof.
\end{proof}

\begin{rmk}
\label{rmk:KL for f}
Analogously, it can be shown that $\Phi$ satisfies the KL property at $\bar x$ with exponent $\theta$ if and only if $f$ satisfies the KL property at $(\bar x, \ys(\bar x))$ with the same exponent $\theta$.
\end{rmk}

The previous theorem motivates us to make the following assumption.
\begin{assumpt}%[\bf \Loj property on $\Phi$]
    \label{ass:value-function-loj}
    The value function $\Phi(x) = \max_y f(x,y)$ satisfies the KL property at each point in $\crit(\Phi):=\{x\in\Rn:  \nabla\Phi(x) = 0\}$ with exponent in $[1/2, 1)$.
\end{assumpt}

Under this assumption and by Theorem \ref{thm:equivalence_FOSP}, if $\beta > \mu^{-1}$, it holds that the Lyapunov function $\hb$ also satisfies the KL property on $\crit(\hb) = \{(x,y) \in \Rn\times\Rm: \dhb(x,y) = 0\}$. This allows us to establish the asymptotic iterate convergence and local convergence rates for Algorithm \ref{alg:line-search-framework}.

%\mbh{Remove the extra condition.}

\begin{theo}
    \label{thm:iter-conv-local-rates}
    Suppose Assumption \ref{Assumption_1}, \ref{ass:search-directions}, \ref{ass:bounded-iterates}, and \ref{ass:value-function-loj} hold and $\beta > \mu^{-1}$. %Furthermore, let $\{(\xk, \yk)\}_k$ be the iterates generated by Algorithm \ref{alg:line-search-framework} with search directions satisfying the additional condition
    % \begin{equation}
    %     \label{eq:KL-addtional-req}
    %     \|\dxk\|^2 \geq c_6\|\dfx(\xk,\ykp)\|^2 - c_7\|\dfy(\xk,\yk)\|^2, \quad \forall k \geq 0,
    % \end{equation}
    % where $c_6 > 0$ and $c_7 \geq 0$.
    Then the iterates $\{(\xk, \yk)\}_k$ generated by Algorithm \ref{alg:line-search-framework} converge to a first-order minimax point $(x^*, y^*)$ of \eqref{Prob_Ori}. Furthermore, it holds that
    \[\|(\xk, \yk) - (x^*, y^*)\|=\begin{cases} 
    O(q^{k}), &\text{if } \theta = 1/2;\\
    O(k^{\frac{1-\theta}{2\theta - 1}}), &\text{if } \theta \in (1/2, 1),
    \end{cases}\]
    where $q \in (0, 1)$ and $\theta$ is the KL exponent of $\hb$ at $(x^*, y^*)$ (or the KL exponent of $\Phi$ at $x^*$).
\end{theo}
\begin{proof}
By Assumption \ref{ass:bounded-iterates}, the sequence $\{\zk\}_k$ must have a limit point $z^* := (x^*, y^*)$. Then, by Theorem \ref{thm:glob-conv-iter-compl} and  part 4 of Proposition \ref{prop:linesearch-terminates-finitely}, $z^*$ must be a first-order minimax point of \eqref{Prob_Ori} and 
\[
\lim_{k\to\infty} \hb(\zk) = \hb(z^*).
\]
Moreover, by Assumption \ref{ass:value-function-loj} and Theorem \ref{thm:KL-equivalence}, $h_\beta$ satisfies the KL property at $z^*$. 

To show $\zk \to z^*$ and the local convergence rates, by \cite[Theorem 3.3 and 3.6]{qian2025convergence}, it suffices to verify \cite[Condition H1 and H2]{qian2025convergence}. 
The update rule of Algorithm \ref{alg:line-search-framework} yields
\begin{equation}
    \label{eq:KL-inter1}
    \begin{aligned}
    \|\zkp - \zk\|^2 
    &= \eta_{y,k}^2 \|\dfy(\zk)\|^2 + \eta_{x,k}^2 \|\dxk\|^2 \\
    &\leq \eta_y^2 \|\dfy(\zk)\|^2 + \eta_x^2 a_1^2 \|\dfx(\zkh)\|^2, 
\end{aligned}
\end{equation}
where the inequality is due to \eqref{eq:search-dir-ass-x}, $\eyk \leq \eta_y$, and $ \exk \leq \eta_x$.  
Then, letting $c_1 := \frac{\tau \gamma_x \min\{b_1 \eyb,b_2 \exb\}}{\max\{\eta_y^2, \eta_x^2 a_1^2\}}$, our line-search condition \eqref{eq:linesearch-condition-x} implies
\begin{align*}
    \hb(\zkp) - \Hk &\leq -\tk\gamma_x[b_1\eyk \|\dfy(\zk)\|^2 + b_2\exk \|\dfx(\zkh)\|^2]\\
    & \leq -\tau \gamma_x[b_1\eyb \|\dfy(\zk)\|^2 + b_2\exb \|\dfx(\zkh)\|^2]\\
    & \leq -c_1 \|\zkp - \zk\|^2,
\end{align*}
where the second inequality is due to $\tk \geq \tau$ and part 3 of Proposition \ref{prop:linesearch-terminates-finitely} and the last inequality is due to \eqref{eq:KL-inter1}. This inequality and \eqref{eq:ZH-weighted-fcn-value} verifies \cite[Condition H1]{qian2025convergence}. 

To verify \cite[Condition H2]{qian2025convergence}, by %\eqref{eq:upper-bd-dhy} and \eqref{eq:upper-bd-dhx},  Lemma \ref{lem:various-gradient-bounds0} and the first inequality of Lemma \ref{lem:various-gradient-bounds1}.
Lemma \ref{lem:various-gradient-bounds0}, it suffices to find $c_2 > 0$ such that 
\[
\|\nabla f(\zkp)\| \leq c_2 \|\zkp - \zk\|.
\]
By \eqref{eq:search-dir-ass-x} and the Cauchy-Schwarz inequality, it holds that
\begin{align*}
    a_2 \|\dfx(\zkh)\|^2 \leq - \inner{\dxk, \dfx(\zkh)} \leq \|\dxk\| \|\dfx(\zkh)\|,
    %\label{eq:KL-inter1.5}
\end{align*}
which implies that $\|\dxk\| \geq a_2 \|\dfx(\zkh)\|$. Invoking this inequality, Proposition \ref{prop:linesearch-terminates-finitely} part 3, and the update rule of Algorithm \ref{alg:line-search-framework} gives 
\begin{equation} 
    \label{eq:KL-inter2}
    \|\zkp - \zk\|^2 
    = \eta_{x,k}^2 \|\dxk\|^2 + \eta_{y,k}^2 \|\dfy(\zk)\|^2
    \geq \exb^2 a_2^2 \|\dfx(\zkh)\|^2 + \eyb^2 \|\dfy(\zk)\|^2.
\end{equation}
Then, by Young's inequality and (local) $\sL$-smoothness of $f$, it holds that
\begin{align*}
    &\|\df(\zkp)\|^2
    = \|\df(\zkp) - \df(\zk) + \df(\zk)\|^2\\
    \leq& 2\sL^2 \|\zkp - \zk\|^2 + 2\|\dfx(\zk)\|^2 + 2\|\dfy(\zk)\|^2\\
    =& 2\sL^2 \|\zkp - \zk\|^2 + 2\|\dfx(\zk) - \dfx(\zkh) + \dfx(\zkh)\|^2 + 2\|\dfy(\zk)\|^2\\
    \leq& 2\sL^2 \|\zkp - \zk\|^2 + 4\sL^2\|\ykp - \yk\|^2 + 4\|\dfx(\zkh)\|^2 + 2\|\dfy(\zk)\|^2\\
    \leq& \bigg(6\sL^2 + \frac{4}{\min\{\exb^2 a_2^2, \eyb^2\}}\bigg) \|\zkp - \zk\|^2,
\end{align*}
where the last line follows from \eqref{eq:KL-inter2}.
Taking $c_2 := \sqrt{6\sL^2 + \frac{4}{\min\{\exb^2 a_2^2, \eyb^2\}}}$ completes the proof.
\end{proof}

\section{Parameter-free GDA with Adaptive Step Sizes and Nonmonotone Line-search} \label{sec:applications}

Our previous analysis are all based on the general Assumption \ref{ass:search-directions} on the search direction $\dxk$. In this section, we will focus on a detailed construction of $\dxk$ that follows the (alternating) GDA method, that is, $\dxk = -\dfx(\xk,\ykp)$. In particular, the previous design and analysis of Algorithm \ref{alg:line-search-framework} are based on the prior knowledge of the strong concavity modulus to determine $\beta$. In this section, we will also design a fully parameter-free GDA method that allows us to progressively estimate and update $\beta$. Combined with the adaptive BB-type (Barzilai-Borwein) step sizes \cite{barzilai1988two} and the ZH-type nonmonotone line-search, we can obtain an efficient realization of our line-search framework. 

\subsection{Adaptive BB Step Sizes}
\label{subsec:BB}
The BB-type step \cite{barzilai1988two}, associated with the gradient descent direction, is an efficient adaptive step size rule that attempts to utilize a minimization problem's local curvature information without computing Hessian.
When generalized to the minimax problem \eqref{Prob_Ori},  the BB-type step size will be associated with the (alternating) GDA update where the search direction is $\dxk = -\dfx(\xk,\ykp)$. For $k \geq 1$, we define 
\begin{equation*}
    \begin{aligned}
        u_{y,k} &:= \yk - \ykm, &\quad v_{y, k} &:= \nabla_y f(\xk, \yk) - \nabla_y f(\xkm,\ykm),\\
         u_{x,k} &:= \xk - \xkm, &\quad  v_{x, k} &:= \nabla_x f(\xk, \ykp) - \nabla_x f(\xkm, \yk).
    \end{aligned}
\end{equation*}
Then, in this scheme, the BB-type initial step sizes $\eta_y$ and $\eta_x$ in each iteration $k$ for the nonmonotone line-search conditions \eqref{eq:linesearch-condition-y} and \eqref{eq:linesearch-condition-x} would be replaced by 
\begin{equation}
\label{eq:BB-step-size}
\begin{aligned}
\eykbb &= \min\Big\{ \eta_{y}^{\max}, \max\Big\{\frac{\|u_{y,k}\|^2}{|\langle u_{y,k}, v_{y,k} \rangle|}, \eta_y^{\min} \Big\}\Big\},\\
\exkbb &= \min\Big\{ \eta_{x}^{\max}, \max\Big\{ \frac{\|u_{x,k}\|^2}{|\langle u_{x,k}, v_{x,k} \rangle|}, \eta_x^{\min} \Big\}\Big\},
% \eykbb &= \min\Bigg\{ \eta_{y}^{\max}, \max\bigg\{\frac{\|\yk - \ykm\|^2}{|\langle \yk - \ykm, \nabla_y f(\xk, \yk) - \nabla_y f(\xkm,\ykm) \rangle|}, \eta_y^{\min} \bigg\}\Bigg\},\\
% \exkbb &= \min\Bigg\{ \eta_{x}^{\max}, \max\bigg\{ \frac{\|\xk - \xkm\|^2}{|\langle \xk - \xkm, \nabla_x f(\xk, \ykp) - \nabla_x f(\xkm, \yk) \rangle|}, \eta_x^{\min} \bigg\}\Bigg\},\\
\end{aligned}
\end{equation}
or
\begin{equation}
\label{eq:BB-step-size2}
\begin{aligned}
\eykbbt &= \min\Big\{ \eta_{y}^{\max}, \max\Big\{\frac{|\langle u_{y,k}, v_{y,k} \rangle|}{\|v_{y,k}\|^2}, \eta_y^{\min} \Big\}\Big\},\\
\exkbbt &= \min\Big\{ \eta_{x}^{\max}, \max\Big\{ \frac{|\langle u_{x,k}, v_{x,k} \rangle|}{\|v_{x,k}\|^2}, \eta_x^{\min} \Big\}\Big\},
\end{aligned}
\end{equation}
where $\eta_{y}^{\max} > \eta_y^{\min} > 0$ and $\eta_{x}^{\max} > \eta_x^{\min} > 0$ regulate the behavior of the BB-type step sizes. Heuristically, the ``long'' BB step sizes \eqref{eq:BB-step-size} produce a larger step and generate fast progress when far away from a minimax point, whereas the ``short'' BB step sizes \eqref{eq:BB-step-size2} produce a smaller step and are more stable near a minimax point. 

Since $\eykbb$, $\exkbb$, $\eykbbt$, and $\eykbbt$ are all bounded above and below, all of the theoretical results in the previous section remain valid when the initial line-search step sizes $\eta_y$ and $\eta_x$ are replaced by the two BB-type stepsizes \eqref{eq:BB-step-size} or \eqref{eq:BB-step-size2}.

\subsection{Parameter-free Gradient Descent-ascent (GDA) with Adaptive Step Sizes} \label{subsec:parameter-free} 
%Following the (alternating) GDA search direction  
%$\dxk = -\dfx(\xk,\ykp)$, 
To resolve the unavailability of the strong concavity modulus $\mu$, we present an adaptive method for estimating the parameter $\beta$, %strong concavity modulus $\mu$, 
and introduce a fully parameter-free GDA method Algorithm \ref{alg:parameter-free-GDA} with robust theoretical guarantees.

\begin{algorithm}[t]
\caption{Parameter-free GDA with Adaptive Step Sizes and Nonmonotone Line-search}
\label{alg:parameter-free-GDA}
\begin{algorithmic}[1]
\State{\textbf{Initialization:} $x_0 \in \Rn$, $y_0 \in \Rm$, $\beta_0 > 0$, $c>0$, $\eta_{y}^{\max} > \eta_y^{\min} > 0$, $\eta_{x}^{\max} > \eta_x^{\min} > 0$, $\alpha \in (0,1)$, $0 < \gamma_x < \gamma_y < 1$, $\{\tk\}_k \subseteq [\tau, 1]$ with $\tau \in (0,1]$, $k=0$, $F_0 = f(x_0, y_0)$, and $G_0 = \|\nabla_y f(x_0, y_0)\|^2$.}
\While{termination criteria not met}
\While{$\inner{\dhyk(\xk, \yk), \dfy(\xk, \yk)} > -c \|\dfy(\xk, \yk)\|^2$}
\State{$\bk \gets 2\bk$.}
\EndWhile
\State{$H_k \gets F_k + \beta_k G_k / 2$ and $\Xi_k \gets \max\{\Hk,\hbk(\xk,\yk)\}$}
\State{Compute $\eta_y = \eykbb$ or $\eykbbt$ by \eqref{eq:BB-step-size} or \eqref{eq:BB-step-size2}. (If $k = 0$, set $\eta_y = \eta_{y}^{\max}$.)}
\State{Choose $\eyk$ to be the largest element in $\{\eta_y \alpha^n: n = 0,1,2,\ldots\}$ satisfying
\begin{equation} \label{eq:GDA-linesearch-condition-y}
    \hbk(\xk, \yk + \eyk \dfy(\xk, \yk)) \leq 
    %\max\{\Hk,\hbk(\xk,\yk)\} 
    \Xi_k-\gamma_y c \eyk \|\nabla_y f(\xk, \yk)\|^2.
\end{equation}}
\vspace{-4mm}
\State{$\ykp \gets \yk + \eyk \dfy(\xk, \yk)$.}
\State{Compute $\eta_x = \exkbb$ or $\exkbbt$ by \eqref{eq:BB-step-size} or \eqref{eq:BB-step-size2}. (If $k = 0$, set $\exkbb = \eta_{x}^{\max}$.)}
\State{Choose $\exk$ to be the largest element in $\{\eta_x \alpha^n: n = 0,1,2,\ldots\}$ satisfying
\begin{equation}  \label{eq:GDA-linesearch-condition-x}
\begin{split}
    &\hbk(\xk - \exk \dfx(\xk, \ykp), \ykp) \leq 
    %\max\{\Hk,\hbk(\xk,\yk)\} 
    \Xi_k- \\&\hspace{2.5cm}\gamma_x \Big( c\eyk \|\nabla_y f(\xk, \yk)\|^2 + \frac{\exk}{2} \|\nabla_x f(\xk, \ykp)\|^2  \Big).
\end{split}
\end{equation}}
\vspace{-2mm}
\State{$\xkp \gets \xk - \exk \dfx(\xk, \ykp)$.}
\State{$F_{k+1} \gets (1-\tk) F_k + \tk f(\xkp, \ykp)$.}
\State{$G_{k+1} \gets (1-\tk) G_k + \tk \|\dfy(\xkp, \ykp)\|^2$.}
\State{$\bkp \gets \bk$ and $k \gets k+1$.}
\EndWhile
\end{algorithmic}
\end{algorithm}

In this algorithm, the initial $\beta_0$ can be estimated from the following equation
\[
\beta_0 = \frac{\|y-y'\|^2}{2[f(x,y) - f(x,y') + \inner{\dfy(x,y), y'-y}]} \in \bigg(0, \frac{1}{\mu}\bigg],
\]
for arbitrary $x \in \Rn$ and $y' \neq y \in \Rm$. Furthermore, although step 3 of this algorithm requires the computation of $\dhyk(\xk, \yk)$, which involves a Hessian-vector product, through modern automatic differentiation techniques, it can be efficiently computed at cost comparable with a gradient evaluation of $f$. In the following, we show that the while loop in step 3--5 terminate finitely.

\begin{lem}
    \label{lem:beta-k-constant}
    Suppose Assumption \ref{Assumption_1} holds. Then, step 3--5 in Algorithm \ref{alg:parameter-free-GDA} would terminate finitely. Moreover, there exists $\bar \beta \in [\beta_0, 2(c+1)\mu^{-1}]$ and $K \geq 0$ such that $\bk = \bar \beta,$ for all $ k \geq K$.
\end{lem}
\begin{proof}
    By the first inequality in Lemma \ref{lem:various-gradient-bounds2}, %\eqref{eq:search-dir-bd-gda-y}, 
    the inequality in step 3 must be violated for any $\bk \geq (c+1)\mu^{-1}$ and thus the while-loop in step 3--5 of Algorithm \ref{alg:parameter-free-GDA} would terminate finitely (since every time the inequality in step 3 is violated, $\bk$ would be doubled). Furthermore, it holds that $\bk \leq 2(c+1)\mu^{-1}$ for all $k \geq 0$ and thus the nondecreasing sequence $\{\bk\}_k$ must reach its limit for all sufficiently large $k$ (since the increase in value for $\{\bk\}_k$ can happen at most finitely many times).  
\end{proof}
Before the values of $\beta_k$ stabilizes (i.e., for iteration $k < K$), it may happen that $\beta_k > \beta_{k-1}$. Such abrupt increases can violate the key inequality $H_k \geq h_{\beta_k}(\xk,\yk)$, which is crucial for the correctness of the nonmonotone line-search procedure. To address this issue, we incorporate additional safeguards (the extra $\{\Xi_k\}_k$ terms) in the line-search conditions \eqref{eq:GDA-linesearch-condition-y} and \eqref{eq:GDA-linesearch-condition-x}. Moreover, we maintain the weighted averages of function values and gradient norms separately as $\{F_k\}_k$ and $\{G_k\}_k$, and combine them into $\{H_k\}_k$ using the most recent value of $\beta_k$ (see Step 6 of Algorithm \ref{alg:parameter-free-GDA}).

To establish theoretical guarantees for the parameter-free GDA method Algorithm \ref{alg:parameter-free-GDA}, we need the following bounded iterates assumption.

\begin{assumpt}
    \label{ass:bounded-iterates2}
    The iterates $\{(\xk, \yk)\}_k$ generated by Algorithm \ref{alg:parameter-free-GDA} are bounded. That is, there exists $M > 0$ such that $(\xk, \yk) \in \cB_M(0) \subseteq \Rn \times \Rm$, for all $k \geq 0$.
\end{assumpt}
\begin{rmk}
    \label{rmk:bdd-iter-2}
    The line-search scheme \eqref{eq:GDA-linesearch-condition-y} and \eqref{eq:GDA-linesearch-condition-x} guarantees \[(\xk, \yk) \in [h_{\bar\beta}(x,y) \leq H_{K}] =: \Omega_{K}, \quad \forall k \geq K.\] So, Assumption \ref{ass:bounded-iterates2} can be ensured if $\Omega_K$ is bounded. By Lemma \ref{lem:coercivity-of-h}, if the value function $\Phi$ is coercive (or the level set $[\Phi(x) \leq H_K]$ is bounded), $\Omega_K$ is bounded when $\bar \beta > \mu^{-1}$; if additionally we have $\sup_k \|\dfy(\xk, \yk)\| < \infty$, the boundedness of $\Omega_K$ can be ensured even if $\bar \beta \leq \mu^{-1}$. One can manually enforce $\sup_k \|\dfy(\xk, \yk)\| < \infty$ by adding the following two conditions
    \begin{equation*}
    %\label{eq:enforce-bdd-grad}
    \begin{gathered}
        \|\dfy(\xk,\yk + \eyk\dfy(\xk,\yk))\|^2 \leq \Gamma^2 - \frac{\gamma_y(c+1)\eyk}{\bk} \|\dfy(\xk,\yk)\|^2,\\
        \|\dfy(\xk-\exk\dfx(\xk,\ykp),\ykp)\|^2 \leq \Gamma^2,
    \end{gathered}
    \end{equation*}
    in the line-search conditions \eqref{eq:GDA-linesearch-condition-y} and \eqref{eq:GDA-linesearch-condition-x}, respectively. Here, we can take $\Gamma = \Gamma_0\|\nabla_y f(x_0, y_0)\| + \Gamma_1$ with arbitrary $\Gamma_0 \geq 1$ and $\Gamma_1 > 0$, justifying the validity of our assumption. 
    %It is easy to check that the backtracking procedures would terminate finitely for the additional conditions in \eqref{eq:enforce-bdd-grad}.
\end{rmk}

With the bounded iterates assumption, we present a proposition for Algorithm \ref{alg:parameter-free-GDA} that is very similar with Proposition \ref{prop:linesearch-terminates-finitely} for Algorithm \ref{alg:line-search-framework}.

\begin{prop}
    \label{prop:parameter-free-basics}
    Given Assumption \ref{Assumption_1} and \ref{ass:bounded-iterates2}, and let $K$ be as  defined in Lemma \ref{lem:beta-k-constant}, then it holds that 
    \begin{enumerate}
        \item The backtracking procedures in \eqref{eq:GDA-linesearch-condition-y} and \eqref{eq:GDA-linesearch-condition-x} terminate finitely for each iteration $k \geq 0$.
        \item There exist $\eyb',\exb' > 0$ such that
        $\eyk \geq \eyb'$ and $\exk \geq \exb'$, $ \forall k \geq K.$
        \item The sequence $\{H_k\}_k$ satisfies $H_k \geq \hbk(\xk, \yk)$ and is nonincreasing for all $k \geq K$. Moreover, there exists a constant $\tilde H$ such that 
        \[
        \lim_{k\to\infty} H_k = \lim_{k\to\infty} \hbk(\xk, \yk) = \tilde H.
        \]
    \end{enumerate}
\end{prop}
\begin{proof}
    Observe that step 3--5 in Algorithm \ref{alg:parameter-free-GDA} guarantees 
    \[
    \inner{\dhyk(\zk), \dfy(\zk)} \leq -c \|\dfy(\zk)\|^2.
    \]
    Moreover, since $\dxk = -\dfx(\zkh)$, by direct computations and the first inequality in Lemma \ref{lem:various-gradient-bounds1} %\eqref{eq:search-dir-bd-gda-x} 
    (replacing $\beta$ by $\bk$), \eqref{eq:search-dir-ass-x} and \eqref{eq:search-dir-requi-x} are satisfied with the following parameters (for $k < K$): $
    a_1=a_2= 1, b_2 = 1/2, \text{ and } b_3 = \beta_k^2 \sL^2/2.$
    Hence, following the same line of arguments as the proof of Proposition \ref{prop:linesearch-terminates-finitely} part 2, the line-search procedures in \eqref{eq:linesearch-condition-y} and \eqref{eq:linesearch-condition-x} would terminate finitely for any iteration $k < K$.
    
    For $k \geq K$, by Lemma \ref{lem:beta-k-constant}, $\beta_k \equiv \bar \beta$. Since the iterates $\{\zk\}_k$ are assumed to be bounded, $\{h_{\bar \beta}(\zk)\}_k$ must be bounded below. As a result, %and \eqref{eq:search-dir-bd-gda-x} (replacing $\beta$ by $\bar \beta$) in Lemma \ref{lem:various-gradient-bounds}, 
    this rest of proposition can be proved in an analogous manner to Proposition \ref{prop:linesearch-terminates-finitely}.
\end{proof}

%Before giving the main theoretical result on Algorithm \ref{alg:parameter-free-GDA}, we first establish a key proposition on the finite termination of the backtracking procedures in line-search and convergence of function values.

% \begin{prop}
%     Suppose Assumption \ref{Assumption_1} and \ref{ass:bounded-iterates2} hold and let $\{(\xk, \yk)\}_k$ be the iterates generated by Algorithm \ref{alg:parameter-free-GDA}. The following statements are true.
% \end{prop}
% \begin{proof}
%     Since the iterates $\{\xk, \yk\}_k$ is bounded, $H_k = h_{\bar \beta}(\xk, \yk)$ must be bounded below for $k \geq K$. The rest of part 2 can be proved analougously to Proposition \ref{prop:linesearch-terminates-finitely}, part 1 and 4.
% \end{proof}

We now present the main theoretical result on our parameter-free GDA method Algorithm \ref{alg:parameter-free-GDA}.

\begin{theo}
    \label{thm:parameter-free-GDA}
    Suppose Assumption \ref{Assumption_1} and \ref{ass:bounded-iterates2} hold and let $\{(\xk, \yk)\}_k$ be the iterates generated by Algorithm \ref{alg:parameter-free-GDA}. Then the following statements are true.
    \begin{enumerate}
        \item It holds that 
        $\nabla f(\xk, \yk) \to 0$ and $\nabla \Phi(\xk) \to 0,$
        as $k \to \infty$. %In particular, every accumulation point of $\{(\xk, \yk)\}_k$ must be a first-order minimax point of \eqref{Prob_Ori}. 
        Furthermore, it holds that %for $T \geq K$,
        \[
        \min_{k=0,1,\ldots,T} \|\nabla f(\xk, \yk)\| = O(T^{-1/2}) \qaq \min_{k=0,1,\ldots,T} \|\nabla \Phi(\xk)\| = O(T^{-1/2}).
        \]
        \item If additionally Assumption \ref{ass:value-function-loj} holds, then the iterates $\{(\xk, \yk)\}_k$ converges to a first-order minimax point $(x^*, y^*)$ of \eqref{Prob_Ori}. Furthermore, it holds that
        \[\|(\xk, \yk) - (x^*, y^*)\|=\begin{cases} 
        O(q^{k}), &\text{if } \theta = 1/2,\\
        O(k^{\frac{1-\theta}{2\theta - 1}}), &\text{if } \theta \in (1/2, 1),
        \end{cases}\]
        where $q \in (0, 1)$ and $\theta$ is the KL exponent of the value function $\Phi$ at $x^*$.
    \end{enumerate}
\end{theo}
\begin{proof}
    Utilizing Proposition \ref{prop:parameter-free-basics}, Part 1 of this theorem can be proved following the same line of arguments as the proof of Theorem \ref{thm:glob-conv-iter-compl}. For part 2, by Assumption \ref{ass:bounded-iterates2}, the sequence $\{\zk\}_k$ has a limit point $z^* := (x^*, y^*)$. Then, by part 1 of this theorem and part 4 of Proposition \ref{prop:parameter-free-basics}, $z^*$ must be a first-order minimax point of \eqref{Prob_Ori} and 
    \[
    \lim_{k\to\infty} h_{\bar \beta}(\zk) = h_{\bar \beta}(z^*).
    \]
    Moreover, following the same arguments as the proof of Theorem \ref{thm:iter-conv-local-rates}, it can be shown that there exists $\hat c_1, \hat c_2 > 0$ such that
    \begin{align}
        \hat c_1 \|\zkp - \zk\|^2 &\leq \Hk - h_{\bar \beta}(\zkp), \quad \forall k \geq K, \label{eq:KL-GDA-inter0}\\ 
        \|\df(\zkp)\| &\leq \hat c_2 \|\zkp - \zk\|, \quad \forall k \geq K. \label{eq:KL-GDA-inter1}
    \end{align}
    Inequality \eqref{eq:KL-GDA-inter0} and the update of $H_k$ when $ k \geq K$: 
    \[\Hkp = (1-\tk)\Hk + \tk h_{\bar \beta}(\zkp), \quad \tk \in [\tau, 1] \text{ for some } \tau \in (0, 1],\] 
    still verify \cite[Condition H1]{qian2025convergence}. However, due to the possibility that $\bar \beta \le \mu^{-1}$, we cannot directly verify  \cite[Condition H2]{qian2025convergence} and the KL property might not hold for $h_{\bar \beta}$ at $z^*$. Fortunately, by Assumption \ref{ass:value-function-loj} and Theorem \ref{thm:KL-equivalence}, $h_{\bar\beta}$ satisfies the quasi-KL property \eqref{eq:quasi-Loj} at $z^*$ and there exists a small neighborhood $U(z^*)$ around $z^*$ s.t. 
    \begin{equation} \label{eq:quasi-Loj2}
        \varphi'(|h_{\bar \beta}(z) - h_{\bar \beta}(z^*)|) \|\nabla f(z)\| \geq 1, \quad \forall z \in U(z^*),
    \end{equation}
    where $\varphi: \bb{R} \to \bb{R}$ with $\varphi(s) = C_{h_{\bar \beta}}(1-\theta)s^{1-\theta}$ is the desingularization function in the context of Kurdyka-\Loj (KL) property \cite{kur98, AttBol09, AttBolRedSou10}.
    Combining \eqref{eq:KL-GDA-inter1} and \eqref{eq:quasi-Loj2} yields that for $\zk \in U(z^*)$,
    \[
    \hat c_2 \|\zk - z_{k-1}\| \cdot \varphi'(|h_{\bar \beta}(\zk) - h_{\bar \beta}(z^*)|) \geq 1
    \iff
    \hat c_2 C_{h_{\bar \beta}} \|\zk-z_{k-1}\| \geq |h_{\bar \beta}(\zk) - h_{\bar \beta}(z^*)|^{\theta}
    \]
    which verifies \cite[Inequality (3.4) and (3.13)]{qian2025convergence}. Since these two inequalities are the only places where the the KL property and \cite[Condition H2]{qian2025convergence} are invoked in the proof of \cite[Theorem 3.3 and 3.6]{qian2025convergence}, by following their proof, one can still derive $\zk \to z^*$ and the corresponding local rates.
\end{proof}

\section{Numerical Experiments}
\label{sec:numerical experiments}
In this section, we present a preliminary numerical experiment on a robust nonlinear
regression problem to validate the performance of our proposed Algorithm \ref{alg:line-search-framework} and \ref{alg:parameter-free-GDA}. More specifically, given a labeled 
dataset $\{(w_i, v_i)\}_{i=1}^N \subseteq \mathbb{R}^d \times \mathbb{R}$, we consider the following NC-SC minimax formulation:
\begin{equation}
    \min_{x \in \mathbb{R}^d}\;
    \max_{\{y_i\}_{i=1}^N \subseteq \mathbb{R}^d}
    f(x,\{y_i\}_{i=1}^N)
    :=
    \frac{1}{N}\sum_{i=1}^N 
    \left[
        \phi\!\left(\langle w_i + y_i, x\rangle - v_i\right) + \frac{\rho_x}{2} \|x\|^2
        - \frac{\rho_y}{2}\|y_i\|^2
    \right],
    \label{eq:robust-regression-minimax}
\end{equation}
where $\phi(\theta) = \theta^2/(1+\theta^2)$ is the smooth biweight loss function 
\cite{BeatonTukey1974,CarmonDuchiHinderSidford2017,LeeWright2020} and $\rho_x\|x\|^2/2$ is a weight-decay regularization term. To promote robustness against outliers and corrupted features, we 
introduce an adversarial perturbation $y_i$ for each data point $w_i$, and regularize its 
magnitude through a quadratic penalty. This adversarially perturbed regression model has 
been widely used in robust learning; see, e.g, \cite{sinha2017certifying,lin2020gradient,ribeiro2023regularization,wang2024efficient}. All numerical experiments %in this section 
are implemented on a server equipped with an Apple M1 Max CPU, running Python 3.13.5, Numpy 2.1.3, and SciPy 1.15.3.

For comparison, we include the two-timescale gradient descent-ascent (TTGDA) 
method \cite{lin2020gradient,lin2025two}, a modern variant of the GDA method for solving NC-SC minimax problems. %TTGDA updates dual variables $y_i$'s with a significantly larger (constant) step size compared to the primal variable to ensure stability. 
Furthermore, we also implement the L-BFGS-B method and gradient descent with BB stepsize directly on the minimization problem \eqref{Prob_Min}. In the following, we present the details of all tested algorithms and datasets. 

\noindent \textbf{TTGDA} \cite{lin2020gradient,lin2025two}: The step sizes $\eta_y$ and $\eta_x$ are chosen from the best combinations in $\eta_y \in \{0.001, 0.005,$ $ 0.01, 0.05, 0.1\}$ and $\eta_x = \theta \eta_y$ with $\theta \in \{0.001, 0.01, 0.1\}$. It is worth noting that the step sizes with convergence guarantees in \cite[Theorem 17]{lin2025two}, %i.e., $\eta_y \sim 1/\sL$ and $\eta_x \sim 1/(\sL \kappa^2)$, where $\kappa = \sL/\mu$, 
are much smaller and lead to slow convergence. \vspace{0.1cm}
    
\noindent \textbf{L-BFGS-B}: The "L-BFGS-B" solver in the SciPy.minimize method applied to the minimization problem \eqref{Prob_Min} with $\beta = 2/\mu$ and $\mu = (\rho_y - 2)/N$.\vspace{0.1cm}

\noindent\textbf{GD-BB}: Gradient descent with nonmonotone line-search and BB step sizes directly applied to \eqref{Prob_Min} with $\beta = 2/\mu$ and $\mu = (\rho_y - 2)/N$. The backtracking factor is $\alpha = 0.5$, the mixing parameter in ZH-type nonmontone line-search is set to $\tau_k \equiv 10^{-3}$, and the descent factor is set to $\gamma = 10^{-4}$. \vspace{0.1cm}

\noindent\textbf{GDA-LS}: Algorithm \ref{alg:line-search-framework} with monotone line-search and GDA search direction. %$\dxk = \nabla_x f(\xk, \ykp)$. Furthermore, to compare our proposed line-search framework with 
The parameters are set as $\mu = (\rho_y - 2)/N$, $\beta = 2/\mu$, $\gamma_y = 10^{-5}$, $\gamma_x = 10^{-12}$, $\alpha = 0.5$, $\tau_k \equiv 1$, and the initial step sizes of the line-search, $\eta_x$ and $\eta_y$, are chosen from the best combinations in \{0.005, 0.001, 0.05, 0.01, 0.5, 0.1, 1\}. \vspace{0.1cm}

\noindent \textbf{GDA-BB}: Algorithm \ref{alg:parameter-free-GDA} with correct $\beta_0 = 2/\mu$ and skipping steps 3--5. We clip the BB step sizes using $\eta_x^{\max} = \eta_y^{\max} = 10^6$ and $\eta_x^{\min} = \eta_y^{\min} = 10^{-6}$. The parameters for the nonmonotone line-search are set as $\gamma_y = 10^{-5}$, $\gamma_x = 10^{-12}$, $\alpha = 0.5$, $\tau_k \equiv 10^{-3}$. \vspace{0.1cm}

\noindent \textbf{GDA-PF}: Algorithm \ref{alg:parameter-free-GDA} with $\beta_0 = 1$. The rest of the parameter choices are the same as GDA-BB. Step 3 to 5 in Algorithm \ref{alg:parameter-free-GDA} are activated every 20 iterations. \vspace{0.1cm}

\noindent\textbf{Synthetic datasets:} In all the test instances with synthetic datasets, the data points $(w_i,v_i)$ are independently generated from the standard Gaussian distribution, i.e., $w_i \sim N(0, I_d)$ and $y_i \sim N(0,1)$. Furthermore, the initial points are set to the origin and the termination criterion is $\|\nabla f\|\le 10^{-7}$. \vspace{0.1cm}

\noindent\textbf{Real-world dataset:} We work with the E2006-tfidf dataset in the LIBSVM library\footnote{See \url{https://www.csie.ntu.edu.tw/~cjlin/libsvmtools/datasets/regression.html}.}. Initially the dataset has 16087 data points and 150360 features. We perform random projection\footnote{We use sklearn.random\_projection.SparseRandomProjection with epsilon set to 0.3.} on this sparse dataset and reduce the number of features to 1076. The initial points are set to all-one vectors and the termination criterion is $\|\nabla f\|\le 10^{-7}$.

\begin{table}[h]
\centering
\setlength{\tabcolsep}{4pt}
\begin{tabular}{c|c|ccccccccc}
\toprule
 & alg. & iter. & f. ev. & g. ev. & hvp & $f$ & $\|\nabla_x f\|$ & $\|\nabla_y f\|$ & $\|\nabla \Phi\|$ & time (s) \\
\midrule

\multirow{6}{*}{
\begin{tabular}{@{}l@{}}
Synthetic\\
$d = 200$ \\
$N = 300$ \\
$\rho_x = 0.1$ \\
$\rho_y = 10$
\end{tabular}
}
  & TTGDA    & 9052 & 0 & 18104 & 0 & 5.29e-2 & 9.94e-8 & 1.05e-8 & 9.84e-8 & 38.86 \\
  & L-BFGS-B & 102 & 120 & 360 & 240 & 5.29e-2 & 6.58e-7 & 1.51e-7 & 6.37e-7 & 9.74 \\
  & GD-BB    & 76 & 151 & 303 & 152 & 5.29e-2 & 5.70e-8 & 2.46e-8 & 5.49e-8  & 1.60 \\
  & GDA-LS   & 2123 & 14707 & 16833 & 0 & 5.29e-2 & 9.91e-8 & 7.53e-9 & 9.87e-8  & 48.34 \\
  & GDA-BB & 104 & 350 & 456 & 0 & 5.29e-2 & 7.79e-8 & 1.86e-8 & 7.62e-8 & 1.64 \\
  & GDA-PF & 134 & 442 & 584 & 6 & 5.29e-2 & 3.65e-8 & 1.04e-8 & 3.56e-8 & 3.58 \\
\midrule

\multirow{6}{*}{
\begin{tabular}{@{}l@{}}
Synthetic\\
$d = 1000$ \\
$N = 1500$ \\
$\rho_x = 0.5$ \\
$\rho_y = 50$
\end{tabular}
}
  & TTGDA & 31422 & 0 & 62844 & 0 & 6.58e-2 & 3.39e-8 & 9.41e-8 & 3.80e-8 & 3772.60 \\
  & L-BFGS-B & 80 & 89 & 267 & 178 & 6.58e-2 & 1.90e-6 & 6.63e-9 & 1.90e-6 & 129.55 \\
  & GD-BB    & 118 & 235 & 471 & 236 & 6.58e-2 & 2.04e-8 & 3.10e-9 & 2.04e-8 & 75.31 \\
  & GDA-LS   & 1633 & 12228 & 13863 & 0 & 6.58e-2 & 9.74e-8 & 1.78e-9 & 9.73e-8 & 1272.96 \\
  & GDA-BB   & 53 & 202 & 257 & 0 & 6.58e-2 & 8.74e-8 & 2.64e-8 & 8.56e-8 & 24.39 \\
  & GDA-PF   & 57 & 208 & 273 & 6 & 6.58e-2 & 3.36e-8 & 5.23e-9 & 3.36e-8 & 48.40 \\
\midrule

\multirow{6}{*}{
\begin{tabular}{@{}l@{}}
Synthetic\\
$d = 2000$ \\
$N = 3000$ \\
$\rho_x = 1$ \\
$\rho_y = 100$
\end{tabular}
}
  & TTGDA    & 25333 & 0 & 50666 & 0 & 6.82e-2 & 1.00e-7 & 3.33e-10 & 1.00e-7 & 12464.24 \\
  & L-BFGS-B & 32 & 42 & 126 & 84 & 6.82e-2 & 2.98e-7 &  2.76e-7 & 2.58e-7 & 206.72\\
  & GD-BB    & 114 & 229 & 457 & 228 & 6.82e-2 & 3.34e-8 & 7.86e-9 & 3.33e-8 & 314.88 \\
  & GDA-LS   & 6178 & 22899 & 29079 & 0 & 6.82e-2 & 8.73e-8 & 4.00e-8 & 8.76e-8 & 10163.41 \\
  & GDA-BB   & 39 & 156 & 197 & 0 & 6.82e-2 & 3.07e-8 & 4.36e-9 & 3.16e-8 &  77.54 \\
  & GDA-PF   & 43 & 167 & 218 & 6 & 6.82e-2 & 7.33e-8 & 3.25e-8 & 7.32e-8 & 150.03 \\
\midrule

\multirow{6}{*}{
\begin{tabular}{@{}l@{}}
Real-world\\
$d = 1076$ \\
$N = 16087$ \\
$\rho_x = 1$ \\
$\rho_y = 200$
\end{tabular}
}
  & TTGDA    & 16120 & 0 & 32240 & 0 & 9.06e-1 & 3.44e-9 & 9.98e-8 & 4.90e-9 & 22536.81 \\
  & L-BFGS-B & 19 & 24 & 72 & 48 & 9.06e-1 & 1.25e-9 &  6.89e-9 & 1.25e-9 & 268.24\\
  & GD-BB    & 59 & 117 & 235 & 118 & 9.06e-1 & 5.67e-8 & 3.30e-10 & 5.67e-8 & 436.06 \\
  & GDA-LS   & 1605 & 4812 & 6419 & 0 & 9.06e-1 & 7.44e-15 & 9.97e-8 & 2.87e-10 & 5985.39 \\
  & GDA-BB   & 12 & 33 & 47 & 0 & 9.06e-1 & 3.65e-9 & 4.83e-9 & 3.63e-9 & 48.60  \\
  & GDA-PF   & 12 & 33 & 55 & 8 & 9.06e-1 & 3.65e-9 & 4.83e-9 & 3.63e-9 & 125.45 \\
\bottomrule
\end{tabular}
\caption{Numerical results on solving \eqref{eq:robust-regression-minimax}. The numerical experiments are run on four different datasets, with the first three instances being synthetic datasets of different dimensions and the last being a real-world dataset, as specified in the first column. The abbreviations, "f. ev., g. ev.," and "hvp", stand for the total number of function evaluations, gradient evalutions and hessian-vector products, respectively. The "time (s)" stands for CPU time (seconds).}
\label{tab:comparison}
\end{table}

The numerical results are summarized in Table~\ref{tab:comparison}. Thanks to its ability to adapt to local geometry and accept larger step sizes, our GDA-LS method outperforms TTGDA, except for the smallest synthetic test instance, where the computational overhead in function evaluation dominates. Furthermore, the effectiveness of our line-search framework is markedly enhanced when combined with BB step sizes. Compared to directly applying L-BFGS-B or GD-BB to the minimization reformulation \eqref{Prob_Min}, our GDA-BB method achieves superior performance due to its avoidance of Hessian–vector products, resulting in lower per-iteration cost. Its parameter-free variant, GDA-PF, incurs a modest additional overhead for estimating $\beta$, yet still delivers outstanding empirical performance. 

\begin{table}[h]
\centering
\setlength{\tabcolsep}{4pt}
\begin{tabular}{c|c|ccccccccc}
\toprule
 & alg. & iter. & f. ev. & g. ev. & hvp & $f$ & $\|\nabla_x f\|$ & $\|\nabla_y f\|$ & $\|\nabla \Phi\|$ & time (s) \\
\midrule

\multirow{4}{*}{
\begin{tabular}{@{}l@{}}
$d = 1000$ \\
$N = 1500$ \\
$\rho_x = 0.01$ \\
$\rho_y = 5$
\end{tabular}
}
  & L-BFGS-B & 558 & 664 & 1992 & 1328 & 4.15e-2 & 2.39e-6 & 1.80e-7 & 2.37e-6 & 976.21 \\
  & GD-BB  & 784 & 1785 & 3353 & 1568 & 4.15e-2 & 1.81e-8 & 4.19e-9 & 1.82e-8 & 528.97 \\
  & GDA-BB & 456 & 1530 & 1988 & 0 & 4.15e-2 & 7.82e-8 & 8.80e-9 & 7.83e-8 & 265.90 \\
  & GDA-PF & 441 & 1492 & 1945 & 10 & 4.15e-2 & 3.15e-8 & 7.01e-8 & 2.92e-8 & 367.38 \\
\midrule

\multirow{4}{*}{
\begin{tabular}{@{}l@{}}
$d = 1000$ \\
$N = 1500$ \\
$\rho_x = 0.01$ \\
$\rho_y = 3$
\end{tabular}
}
  & L-BFGS-B & 864 & 1042 & 3126 & 2084 & 4.55e-2 & 5.96e-7 & 3.26e-8 & 5.96e-7 & 1463.52 \\
  & GD-BB  & 1278 & 2830 & 5386 & 2556 & 4.55e-2 & 8.64e-8 & 2.27e-9 & 8.63e-8 & 818.33 \\
  & GDA-BB & 419 & 1380 & 1801 & 0 & 4.55e-2 & 9.13e-8 & 3.69e-9 & 9.11e-8 & 180.38 \\
  & GDA-PF & 368 & 1238 & 1619 & 11 & 4.55e-2 & 7.64e-8 & 1.09e-8 & 7.63e-8 & 258.88 \\
% \midrule

% \multirow{4}{*}{
% \begin{tabular}{@{}l@{}}
% $d = 1000$ \\
% $N = 1500$ \\
% $\rho_x = 0.01$ \\
% $\rho_y = 2.5$
% \end{tabular}
% }
%   & L-BFGS-B &&&&&&&&& \\
%   & GD-BB  &&&&&&&&& \\
%   & GDA-BB &&&&&&&&& \\
%   & GDA-PF &&&&&&&&& \\
\bottomrule
\end{tabular}
\caption{Numerical results on solving \eqref{eq:robust-regression-minimax} with significantly smaller $\rho_x$ and $\rho_y$. The numerical experiments are run on two different synthetic datasets, as specified in the first column. The abbreviations, "f. ev., g. ev.," and "hvp", stand for the total number of function evaluations, gradient evalutions and hessian-vector products, respectively. The "time (s)" stands for CPU time (seconds).}
\label{tab:comparison2}
\end{table}

To further compare the performance of algorithms based on the minimization reformulation \eqref{Prob_Min}, we apply L-BFGS-B, GD-BB, GDA-BB, and GDA-PF to \eqref{eq:robust-regression-minimax} with smaller $\rho_x$ and $\rho_y$, the numerical results of which are summarized in Table \ref{tab:comparison2}. A smaller $\rho_y$ leads to a smaller strongly convexity moduli $\mu$ and a larger $\beta$ is required for the reformulation \eqref{Prob_Min}. As shown in Table \ref{tab:comparison2}, the alternating GDA directions employed by our proposed methods GDA-BB and GDA-PF can better adapt to the ill-conditioning induced by a significantly larger $\beta$.

\section{Conclusion}
This paper revisits smooth NC–SC minimax optimization \eqref{Prob_Ori} through the lens of an equivalent minimization reformulation \eqref{Prob_Min}. We demonstrated that this reformulation possesses markedly stronger properties than previously understood: it preserves minimax solutions at both first and second order, retains global and local optimality structures, and inherits the KL property. These findings establish it as a principled Lyapunov function for algorithmic design.

Building on this foundation, we proposed a unified line-search framework capable of exploiting local problem geometry while ensuring convergence guarantees. The framework supports adaptive, parameter-free step-size policies, most notably BB step sizes, that have traditionally been difficult to implement in minimax optimization. Our theoretical results provide global convergence, global rate guarantees, and, when KL property holds, sequence convergence and local rates. Empirical evidence confirms our theoretical analysis. 

\bibliographystyle{plain}
\bibliography{references}

@article{hu2024minimization,
  title={A Minimization Approach for Minimax Optimization with Coupled Constraints},
  author={Hu, Xiaoyin and Toh, Kim-Chuan and Wang, Shiwei and Xiao, Nachuan},
  journal={arXiv preprint arXiv:2408.17213},
  year={2024}
}

@article{lin2025two,
  title={Two-timescale gradient descent ascent algorithms for nonconvex minimax optimization},
  author={Lin, Tianyi and Jin, Chi and Jordan, Michael I},
  journal={J. Mach. Learn. Res.},
  volume={26},
  number={11},
  pages={1--45},
  year={2025}
}

@article{lu2025first,
  title={A first-order method for nonconvex-nonconcave minimax problems under a local Kurdyka-{\L}ojasiewicz condition},
  author={Lu, Zhaosong and Wang, Xiangyuan},
  journal={arXiv preprint arXiv:2507.01932},
  year={2025}
}

@article{lu2025first2,
  title={A first-order method for constrained nonconvex--nonconcave minimax problems under a local Kurdyka-{\L}ojasiewicz condition},
  author={Lu, Zhaosong and Wang, Xiangyuan},
  journal={arXiv preprint arXiv:2510.01168},
  year={2025}
}

@article{li2025nonsmooth,
  title={Nonsmooth nonconvex--nonconcave minimax optimization: Primal--dual balancing and iteration complexity analysis},
  author={Li, Jiajin and Zhu, Linglingzhi and So, Anthony Man-Cho},
  journal={Math. Program.},
  pages={1--51},
  year={2025},
  publisher={Springer}
}

@article{zheng2025doubly,
  title={Doubly smoothed optimistic gradients: A universal approach for smooth minimax problems},
  author={Zheng, Taoli and So, Anthony Man-Cho and Li, Jiajin},
  journal={arXiv preprint arXiv:2506.07397},
  year={2025}
}

@article{zheng2023universal,
  title={Universal gradient descent ascent method for nonconvex-nonconcave minimax optimization},
  author={Zheng, Taoli and Zhu, Linglingzhi and So, Anthony Man-Cho and Blanchet, Jos{\'e} and Li, Jiajin},
  journal={Adv. Neural Inf. Process. Syst.},
  volume={36},
  pages={54075--54110},
  year={2023}
}

@article{zheng2022doubly,
  title={Doubly Smoothed GDA for Constrained Nonconvex-Nonconcave Minimax Optimization},
  author={Zheng, Taoli and Zhu, Linglingzhi and So, Anthony Man-Cho and Blanchet, Jos{\'e} and Li, Jiajin},
  journal={arXiv preprint arXiv:2212.12978},
  year={2022}
}

@article{cohen2025alternating,
  title={Alternating and parallel proximal gradient methods for nonsmooth, nonconvex minimax: a unified convergence analysis},
  author={Cohen, Eyal and Teboulle, Marc},
  journal={Math. Oper. Res.},
  volume={50},
  number={1},
  pages={141--168},
  year={2025},
  publisher={INFORMS}
}

@book{Bertsekas2016NonlinearProgramming,
  author    = {Dimitri P. Bertsekas},
  title     = {Nonlinear Programming},
  edition   = {3rd},
  year      = {2016},
  publisher = {Athena Scientific},
  address   = {Belmont, Massachusetts},
  isbn      = {978-1-886529-05-2},
  note      = {Third Edition}
}

@article{grippo1986nonmonotone,
  title={A nonmonotone line search technique for Newton’s method},
  author={Grippo, Luigi and Lampariello, Francesco and Lucidi, Stefano},
  journal={SIAM J. Numer. Anal.},
  volume={23},
  number={4},
  pages={707--716},
  year={1986},
  publisher={SIAM}
}

@article{raydan1997barzilai,
  title={The Barzilai and Borwein gradient method for the large scale unconstrained minimization problem},
  author={Raydan, Marcos},
  journal={SIAM J. Optim.},
  volume={7},
  number={1},
  pages={26--33},
  year={1997},
  publisher={SIAM}
}

@article{grippo1989truncated,
  title={A truncated Newton method with nonmonotone line search for unconstrained optimization},
  author={Grippo, Luigi and Lampariello, Francesco and Lucidi, Stefano},
  journal={J. Optim. Theory Appl.},
  volume={60},
  number={3},
  pages={401--419},
  year={1989},
  publisher={Springer}
}

@article{zhang2004nonmonotone,
  title={A nonmonotone line search technique and its application to unconstrained optimization},
  author={Zhang, Hongchao and Hager, William W},
  journal={SIAM J. Optim.},
  volume={14},
  number={4},
  pages={1043--1056},
  year={2004},
  publisher={SIAM}
}

@article{hager2005new,
  title={A new conjugate gradient method with guaranteed descent and an efficient line search},
  author={Hager, William W and Zhang, Hongchao},
  journal={SIAM J. Optim.},
  volume={16},
  number={1},
  pages={170--192},
  year={2005},
  publisher={SIAM}
}

@article{qian2023convergence,
  title={Convergence of a class of nonmonotone descent methods for kurdyka-{\L}ojasiewicz optimization problems},
  author={Qian, Yitian and Pan, Shaohua},
  journal={SIAM J. Optim.},
  volume={33},
  number={2},
  pages={638--651},
  year={2023},
  publisher={SIAM}
}

@article{qian2025convergence,
  title={Convergence of ZH-Type Nonmonotone Descent Method for Kurdyka-{\L}ojasiewicz Optimization Problems},
  author={Qian, Yitian and Tao, Ting and Pan, Shaohua and Qi, Houduo},
  journal={SIAM J. Optim.},
  volume={35},
  number={2},
  pages={1089--1109},
  year={2025},
  publisher={SIAM}
}

@article{barzilai1988two,
  title={Two-point step size gradient methods},
  author={Barzilai, Jonathan and Borwein, Jonathan M},
  journal={IMA J. Numer. Anal.},
  volume={8},
  number={1},
  pages={141--148},
  year={1988},
  publisher={Oxford University Press}
}

@incollection{lojasiewicz1965ensembles,
    AUTHOR = {{\L}ojasiewicz, S.},
     TITLE = {Une propri\'{e}t\'{e} topologique des sous-ensembles
              analytiques r\'{e}els},
 BOOKTITLE = {Les \'{E}quations aux {D}\'{e}riv\'{e}es {P}artielles},
    SERIES = {Colloq. Internat.},
PUBLISHER = {CNRS},
    VOLUME = {117},
     PAGES = {87--89},
      YEAR = {1963},
   MRCLASS = {26.55 (32.25)},
  MRNUMBER = {160856},
MRREVIEWER = {H.\ A.\ Antosiewicz},
}

@article{kur98,
    author = {Kurdyka, Krzysztof},
    title = {On gradients of functions definable in o-minimal structures},
    journal = {Ann. Inst. Fourier (Grenoble)},
    fjournal = {Universit\'{e} de Grenoble. Annales de l'Institut Fourier},
    volume = {48},
    year = {1998},
    number = {3},
    pages = {769--783},
    issn = {0373-0956},
    mrclass = {03C65 (14P15 26D10 26E05)},
    mrnumber = {1644089},
    mrreviewer = {A. J. Wilkie},
}

@article{absil2005convergence,
	title={Convergence of the iterates of descent methods for analytic cost functions},
	author={Absil, Pierre-Antoine and Mahony, Robert and Andrews, Benjamin},
	journal={SIAM J. Optim.},
	volume={16},
	number={2},
	pages={531--547},
	year={2005},
	publisher={SIAM}
}

@article {AttBolRedSou10,
	AUTHOR = {Attouch, H\'{e}dy and Bolte, J\'{e}r\^{o}me and Redont, Patrick and Soubeyran, Antoine},
	TITLE = {Proximal alternating minimization and projection methods for nonconvex problems: an approach based on the {K}urdyka-{{\L}}ojasiewicz inequality},
	JOURNAL = {Math. Oper. Res.},
	FJOURNAL = {Math. Oper. Res.},
	VOLUME = {35},
	YEAR = {2010},
	NUMBER = {2},
	PAGES = {438--457},
	ISSN = {0364-765X},
	MRCLASS = {90C26 (49J52 65K10)},
	MRNUMBER = {2674728},
}

@article {AttBolSva13,
	AUTHOR = {Attouch, Hedy and Bolte, J\'{e}r\^{o}me and Svaiter, Benar Fux},
	TITLE = {Convergence of descent methods for semi-algebraic and tame problems: proximal algorithms, forward-backward splitting, and regularized {G}auss-{S}eidel methods},
	JOURNAL = {Math. Program.},
	FJOURNAL = {Math. Program.},
	VOLUME = {137},
	YEAR = {2013},
	NUMBER = {1-2, Ser. A},
	PAGES = {91--129},
	ISSN = {0025-5610},
	MRCLASS = {49J53 (47J25 47J30 49M15 65K10 90C30)},
	MRNUMBER = {3010421},
	MRREVIEWER = {C. H. Jeffrey Pang},
}

@article{AttBol09,
    author = {Attouch, Hedy and Bolte, J\'{e}r\^{o}me},
    title = {On the convergence of the proximal algorithm for nonsmooth functions involving analytic features},
    journal = {Math. Program.},
    fjournal = {Mathematical Programming. A Publication of the Mathematical Programming Society},
    volume = {116},
    year = {2009},
    number = {1-2, Ser. B},
    pages = {5--16},
    issn = {0025-5610},
    mrclass = {90C30 (65K05)},
    mrnumber = {2421270},
    mrreviewer = {W. W. Breckner},
}

@article {BolSabTeb14,
	AUTHOR = {Bolte, J\'{e}r\^{o}me and Sabach, Shoham and Teboulle, Marc},
	TITLE = {Proximal alternating linearized minimization for nonconvex and nonsmooth problems},
	JOURNAL = {Math. Program.},
	FJOURNAL = {Mathematical Programming},
	VOLUME = {146},
	YEAR = {2014},
	NUMBER = {1-2, Ser. A},
	PAGES = {459--494},
	ISSN = {0025-5610},
	MRCLASS = {90C26 (90C56)},
	MRNUMBER = {3232623},
	MRREVIEWER = {Yiran He},
}

@book{rockafellar1998variational,
  title={Variational analysis},
  author={Rockafellar, R Tyrrell and Wets, Roger JB},
  year={1998},
  publisher={Springer}
}

@article{cartis2015worst,
  title={Worst-case evaluation complexity of non-monotone gradient-related algorithms for unconstrained optimization},
  author={Cartis, Coralia and Sampaio, Ph R and Toint, Ph L},
  journal={Optimization},
  volume={64},
  number={5},
  pages={1349--1361},
  year={2015},
  publisher={Taylor \& Francis}
}

@article{wolfe1969convergence,
  title={Convergence conditions for ascent methods},
  author={Wolfe, Philip},
  journal={SIAM Rev.},
  volume={11},
  number={2},
  pages={226--235},
  year={1969},
  publisher={SIAM}
}

@article{armijo1966minimization,
  title={Minimization of functions having Lipschitz continuous first partial derivatives},
  author={Armijo, Larry},
  journal={Pacific J. Math.},
  volume={16},
  number={1},
  pages={1--3},
  year={1966},
  publisher={Mathematical Sciences Publishers}
}

@article{wolfe1971convergence,
  title={Convergence conditions for ascent methods. II: Some corrections},
  author={Wolfe, Philip},
  journal={SIAM Rev.},
  volume={13},
  number={2},
  pages={185--188},
  year={1971},
  publisher={SIAM}
}

@article{goldstein1965steepest,
  title={On steepest descent},
  author={Goldstein, Allen A},
  journal={J. Soc. Indust. Appl. Math. Ser. A Control},
  volume={3},
  number={1},
  pages={147--151},
  year={1965},
  publisher={SIAM}
}

@article{grapiglia2021generalized,
  title={A generalized worst-case complexity analysis for non-monotone line searches},
  author={Grapiglia, Geovani N and Sachs, Ekkehard W},
  journal={Numer. Algorithms},
  volume={87},
  number={2},
  pages={779--796},
  year={2021},
  publisher={Springer}
}

@article{lu2017randomized,
  title={A randomized nonmonotone block proximal gradient method for a class of structured nonlinear programming},
  author={Lu, Zhaosong and Xiao, Lin},
  journal={SIAM J. Numer. Anal.},
  volume={55},
  number={6},
  pages={2930--2955},
  year={2017},
  publisher={SIAM}
}

@article{BeatonTukey1974,
  author    = {Beaton, Albert E. and Tukey, John W.},
  title     = {The Fitting of Power Series, Meaning Polynomials, Illustrated on Band-Spectroscopic Data},
  journal   = {Technometrics},
  volume    = {16},
  number    = {2},
  pages     = {147--185},
  year      = {1974},
  publisher = {Taylor \& Francis},
}

@inproceedings{lin2020gradient,
  title={On gradient descent ascent for nonconvex-concave minimax problems},
  author={Lin, Tianyi and Jin, Chi and Jordan, Michael},
  booktitle={International conference on machine learning},
  pages={6083--6093},
  year={2020},
  organization={PMLR}
}

@inproceedings{CarmonDuchiHinderSidford2017,
  title={“Convex until proven guilty”: Dimension-free acceleration of gradient descent on non-convex functions},
  author={Carmon, Yair and Duchi, John C and Hinder, Oliver and Sidford, Aaron},
  booktitle={International conference on machine learning},
  pages={654--663},
  year={2017},
  organization={PMLR}
}

@article{LeeWright2020,
  author    = {Lee, Ching-pei and Wright, Stephen J.},
  title     = {Inexact Variable Metric Stochastic Block-Coordinate Descent for Regularized Optimization},
  journal   = {J. Optim. Theory Appl.},
  volume    = {185},
  number    = {1},
  pages     = {151--187},
  year      = {2020},
  publisher = {Springer},
}

@article{ribeiro2023regularization,
  title={Regularization properties of adversarially-trained linear regression},
  author={Ribeiro, Antonio and Zachariah, Dave and Bach, Francis and Sch{\"o}n, Thomas},
  journal={Adv. Neural Inf. Process. Syst.},
  volume={36},
  pages={23658--23670},
  year={2023}
}

@article{sinha2017certifying,
  title={Certifying some distributional robustness with principled adversarial training},
  author={Sinha, Aman and Namkoong, Hongseok and Volpi, Riccardo and Duchi, John},
  journal={arXiv preprint arXiv:1710.10571},
  year={2017}
}

@article{yang2022nest,
  title={Nest your adaptive algorithm for parameter-agnostic nonconvex minimax optimization},
  author={Yang, Junchi and Li, Xiang and He, Niao},
  journal={Adv. Neural Inf. Process. Syst.},
  volume={35},
  pages={11202--11216},
  year={2022}
}

@inproceedings{li2023tiada,
  title={TiAda: A Time-scale Adaptive Algorithm for Nonconvex Minimax Optimization},
  author={Li, Xiang and Yang, Junchi and He, Niao},
  booktitle={The Eleventh International Conference on Learning Representations (ICLR 2023)},
  year={2023}
}

@inproceedings{huang23a,
  title={Adagda: Faster adaptive gradient descent ascent methods for minimax optimization},
  author={Huang, Feihu and Wu, Xidong and Hu, Zhengmian},
  booktitle={International Conference on Artificial Intelligence and Statistics},
  pages={2365--2389},
  year={2023},
  organization={PMLR}
}

@InProceedings{lin20a,
  title = 	 {Near-Optimal Algorithms for Minimax Optimization},
  author =       {Lin, Tianyi and Jin, Chi and Jordan, Michael I.},
  booktitle = 	 {Proceedings of Thirty Third Conference on Learning Theory},
  pages = 	 {2738--2779},
  year = 	 {2020},
  volume = 	 {125},
  series = 	 {Proceedings of Machine Learning Research},
  month = 	 {09--12 Jul},
  publisher =    {PMLR},
}

@article{wang2024efficient,
  title={Efficient first order method for saddle point problems with higher order smoothness},
  author={Wang, Nuozhou and Zhang, Junyu and Zhang, Shuzhong},
  journal={SIAM J. Optim.},
  volume={34},
  number={4},
  pages={3342--3370},
  year={2024},
  publisher={SIAM}
}

@article{xu2024stochasticgdamethodbacktracking,
      title={A Stochastic GDA Method With Backtracking For Solving Nonconvex (Strongly) Concave Minimax Problems}, 
      author={Qiushui Xu and Xuan Zhang and Necdet Serhat Aybat and Mert Gürbüzbalaban},
      year={2024},
      journal={arXiv preprint arXiv:2403.07806} 
}

@article{zhang2024agdaproximalalternatinggradient,
      title={AGDA+: Proximal Alternating Gradient Descent Ascent Method with a Nonmonotone Adaptive Step-Size Search for Nonconvex Minimax Problems}, 
      author={Xuan Zhang and Qiushui Xu and Necdet Serhat Aybat},
      journal={arXiv preprint arXiv:2406.14371},
      year={2024}
}

@article{bolte2023backtrack,
  title={The backtrack H{\"o}lder gradient method with application to min-max and min-min problems},
  author={Bolte, J{\'e}r{\^o}me and Glaudin, Lilian and Pauwels, Edouard and Serrurier, Mathieu},
  journal={Open J. Math. Optim.},
  volume={4},
  number={8},
  year={2023}
}

@article{li2021complexity,
  title={Complexity lower bounds for nonconvex-strongly-concave min-max optimization},
  author={Li, Haochuan and Tian, Yi and Zhang, Jingzhao and Jadbabaie, Ali},
  journal={Adv. Neural Inf. Process. Syst.},
  volume={34},
  pages={1792--1804},
  year={2021}
}

@inproceedings{zhang2021complexity,
  title={The complexity of nonconvex-strongly-concave minimax optimization},
  author={Zhang, Siqi and Yang, Junchi and Guzm{\'a}n, Crist{\'o}bal and Kiyavash, Negar and He, Niao},
  booktitle={Uncertainty in Artificial Intelligence},
  pages={482--492},
  year={2021},
  organization={PMLR}
}

@article{goodfellow2020generative,
  title={Generative adversarial networks},
  author={Goodfellow, Ian and Pouget-Abadie, Jean and Mirza, Mehdi and Xu, Bing and Warde-Farley, David and Ozair, Sherjil and Courville, Aaron and Bengio, Yoshua},
  journal={Commun. ACM},
  volume={63},
  number={11},
  pages={139--144},
  year={2020},
  publisher={ACM New York, NY, USA}
}

@article{shafieezadeh2015distributionally,
  title={Distributionally robust logistic regression},
  author={Shafieezadeh Abadeh, Soroosh and Mohajerin Esfahani, Peyman M and Kuhn, Daniel},
  journal={Adv. Neural Inf. Process. Syst.},
  volume={28},
  year={2015}
}

@article{xu2023unified,
  title={A unified single-loop alternating gradient projection algorithm for nonconvex--concave and convex--nonconcave minimax problems},
  author={Xu, Zi and Zhang, Huiling and Xu, Yang and Lan, Guanghui},
  journal={Math. Program.},
  volume={201},
  number={1},
  pages={635--706},
  year={2023},
  publisher={Springer}
}

@article{nouiehed2019solving,
  title={Solving a class of non-convex min-max games using iterative first order methods},
  author={Nouiehed, Maher and Sanjabi, Maziar and Huang, Tianjian and Lee, Jason D and Razaviyayn, Meisam},
  journal={Adv. Neural Inf. Process. Syst.},
  volume={32},
  year={2019}
}

@inproceedings{nesterov1983method,
  title={A method for solving the convex programming problem with convergence rate O (1/k2)},
  author={Nesterov, Yurii},
  booktitle={Dokl akad nauk Sssr},
  volume={269},
  pages={543},
  year={1983}
}

@article{rockafellar1976monotone,
  title={Monotone operators and the proximal point algorithm},
  author={Rockafellar, R Tyrrell},
  journal={SIAM J. Control Optim.},
  volume={14},
  number={5},
  pages={877--898},
  year={1976},
  publisher={SIAM}
}

@article{mahdavinia2022tight,
  title={Tight analysis of extra-gradient and optimistic gradient methods for nonconvex minimax problems},
  author={Mahdavinia, Pouria and Deng, Yuyang and Li, Haochuan and Mahdavi, Mehrdad},
  journal={Adv. Neural Inf. Process. Syst.},
  volume={35},
  pages={31213--31225},
  year={2022}
}
\end{document}